\documentclass[aop]{imsart}

\RequirePackage{amsthm,amsmath,amsfonts,amssymb,latexsym,amsfonts,amscd}
\RequirePackage[numbers]{natbib}
\RequirePackage[colorlinks,citecolor=blue,urlcolor=blue]{hyperref}
\RequirePackage{graphicx}
\RequirePackage{subfigure}
\RequirePackage{mathrsfs}
\RequirePackage{footmisc}
\startlocaldefs
\theoremstyle{plain}

\newtheorem{theorem}{Theorem}[section]
\newtheorem{lemma}[theorem]{Lemma}
\newtheorem{corollary}[theorem]{Corollary}
\newtheorem{proposition}[theorem]{Proposition}
\theoremstyle{remark}
\newtheorem{definition}[theorem]{Definition}
\newtheorem{remark}[theorem]{Remark}

\makeatletter
\newcommand*{\rom}[1]{\expandafter\@slowromancap\romannumeral #1@}
\makeatother

\newcommand{\abs}[1]{\left\lvert #1 \right\rvert}

\newcommand{\eqnsection}{
\renewcommand{\theequation}{\thesection.\arabic{equation}}
    \makeatletter
    \csname  @addtoreset\endcsname{equation}{section}
    \makeatother}
\eqnsection

\allowdisplaybreaks[4]
\DeclareFontFamily{U}{mathx}{\hyphenchar\font45}
\DeclareFontShape{U}{mathx}{m}{n}{
      <5> <6> <7> <8> <9> <10>
      <10.95> <12> <14.4> <17.28> <20.74> <24.88>
      mathx10
      }{}
\DeclareSymbolFont{mathx}{U}{mathx}{m}{n}
\DeclareFontSubstitution{U}{mathx}{m}{n}
\DeclareMathAccent{\widecheck}{0}{mathx}{"71}
\DeclareMathAccent{\wideparen}{0}{mathx}{"75}

\DeclareFontFamily{U}{lasy}{}
\DeclareFontShape{U}{lasy}{m}{n}{
  <-5.5> lasy5
  <5.5-6.5> lasy6
  <6.5-7.5> lasy7
  <7.5-8.5> lasy8
  <8.5-9.5> lasy9
  <9.5-> lasy10
}{}

\def\r{{\mathbb R}}
\def\e{{\mathbb E}}
\def\p{{\mathbb P}}
\def\q{{\mathbb Q}}

\def\z{{\mathbb Z}}
\def\N{{\mathbb N}}
\def\T{{\mathbb T}}

\def\d{\, \mathrm{d}}

\def\cL{\mathcal L}
\def\O{\mathcal O}

\def\wL{{\widetilde{\mathcal L} }}

\def\tG{{\widetilde {\mathbb T}}}

\def\cX{\mathcal X}

\def\cZ{\mathcal Z}
\def\bK{\mathbb K}
\def\bL{\mathbf L}

\def\cN{\mathcal N}

\endlocaldefs

\begin{document}

\begin{frontmatter}
\title{Inverting Ray-Knight identities on trees}
\runtitle{Inverting Ray-Knight identities on trees}

\begin{aug}
\author[A]{\fnms{Xiaodan}~\snm{Li}\ead[label=e1]{lixiaodan@mail.shufe.edu.cn}},
\author[B]{\fnms{Yushu}~\snm{Zheng}\ead[label=e2]{yszheng666@gmail.com}}%
\address[A]{Department of Mathematics,
Shanghai University of Finance and Economics, \printead[presep={,\ }]{e1}}

\address[B]{Shanghai Center for Mathematical Sciences,
Fudan University \printead[presep={,\ }]{e2}}
\end{aug}

\begin{abstract}
In this paper, we first introduce the Ray-Knight identity and percolation Ray-Knight identity related to loop soup  with intensity $\alpha (\ge 0)$ on trees. Then we present the inversions of the above identities, which are expressed in terms of  repelling jump processes. In particular, the inversion in the case of $\alpha=0$ gives the conditional law of a continuous-time Markov chain given its local time field. We further show that the fine mesh limits of these repelling jump processes are the self-repelling diffusions \cite{warren98,Aidekon} involved in the inversion of the Ray-Knight identity on the corresponding metric graph. This is a generalization of results in \cite{2016Inverting,lupu2019inverting,LupuEJP657}, where the authors explore the case of $\alpha=1/2$ on a general graph. Our construction is different from \cite{2016Inverting,lupu2019inverting} and based on the link between random networks  and loop soups.
\end{abstract}



\end{frontmatter}
\setcounter{tocdepth}{2}
\tableofcontents
\section{Introduction}\label{introduction}
Imagine a Brownian crook who spent a month in a large metropolis. The number of nights he spent in hotels A, B, C, $\cdots$, etc. is known; but not the order, nor his itinerary. So the only information the police has is total hotel bills. This vivid story is quoted from \cite{warren98}, which is also the paper the name `Brownian burglar' comes from. In \cite{warren98}, Warren and Yor constructed the Brownian burglar to describe the law of reflected Brownian motion on the positive half line conditioned on its local time process (also called  occupation time field). Meanwhile, Aldous \cite{1998BROWNIAN} used the tree structure of the Brownian excursion to show that the genealogy of the conditioned Brownian motion is a time-changed Kingman coalescent. The article \cite{warren98}  can be viewed as a construction of the process in time while \cite{1998BROWNIAN} as a construction in space.
Then a natural question arises. 

(Q1) How can we describe the law of a continuous-time Markov chain (CTMC) conditionally on its occupation time field?

This problem can actually be seen as a special case in a more general class of recovery problems that we are to explain. The occupation time field of a CTMC is considered in the generalized second Ray-Knight theorem, which provides an identity between  the law of the sum of half of a squared Gaussian free field (GFF) with boundary condition $0$ and the occupation time field of an independent Markovian path on one hand, and the law of half of a squared GFF with boundary condition $\sqrt{2u}$ on the other hand (see \cite{Eisenbaum00,Eisenbaum94,2006Markov}). We call this identity a Ray-Knight identity. 
It is well-known that the occupation time field of a loop soup with intensity $\alpha=1/2$ is distributed as half of a squared GFF by Le Jan's isomorphism \cite{le2011markov}. Therefore the generalized second Ray-Knight identity can also be stated using a loop soup with intensity $1/2$.  In the case of a loop soup with arbitrary intensity $\alpha > 0$, an analogous identity holds: adding to the occupation time field of a loop soup with `boundary condition' $0$  the occupation time field of an independent CTMC until local time $u$ gives the distribution of  the occupation time field of a loop soup with `boundary condition' $u$, see Proposition \ref{rayknighttransient} below for a precise statement. We call any such identity a Ray-Knight identity. Inverting the Ray-Knight identity refers to recovering the CTMC conditioned on the total occupation time field.

Vertex-reinforced jump processes (VRJP), conceived by Werner and first studied by Davis and Volkov \cite{2002Continuous,2004Vertex}, are  continuous-time jump processes favouring sites with higher local times. Surprisingly, Sabot and Tarr\`es \cite{2016Inverting} found that a time change of a variant of VRJP provides an inversion of the Ray-Knight identity on a general graph in the case of $\alpha=1/2$. It is natural to wonder whether an analogous description holds for an arbitrary intensity $\alpha$.

(Q2) For any $\alpha>0$, how can we describe the process that inverts the Ray-Knight identity 
related to  loop soup with intensity $\alpha>0$?

Note that (Q1) can be viewed as a special case of (Q2) with $\alpha=0$ if we generalize (Q2). Intuitively, when $\alpha=0$, the external interference of the loop soup disappears. Hence it reduces to extracting the CTMC from its own occupation time field. Another equivalent interpretation of (Q2) is to recover the loop soup with intensity $\alpha$ conditioned on its local time field. 

In \cite{lupu2016loop}, Lupu gave a `signed version' of Le Jan's isomorphism 
where the loop soup at intensity $\alpha=1/2$ is naturally coupled with a signed GFF. In \cite[Theorem 8]{lupu2019inverting}, Lupu, Sabot and Tarr\`es gave the corresponding version of the Ray-Knight identity (we call it a percolation Ray-Knight identity). Besides the identity of local time fields, by adding a percolation along with the Markovian path and finally sampling signs in every connected component of the percolation, one can start with a GFF with boundary condition $0$ and end up getting a GFF with boundary condition $\sqrt{2u}$. The inversion of the signed isomorphism is carried out in that paper and involves another type of self-interacting process \cite[\S3]{lupu2019inverting}.
It leads to the following question. 

(Q3) Can we generalize the percolation Ray-Knight identity to the case of loop soup with intensity $\alpha>0$? If so, how can we describe the process that inverts the percolation Ray-Knight identity? 

The analogous problems can also be considered for Brownian motion and Brownian loop soup. 
In \cite{LupuEJP657}, Lupu, Sabot and Tarr\`es constructed a  self-repelling diffusion out of a divergent Bass-Burdzy flow which inverts the Ray-Knight identity related to  GFF on the line and showed that the self-repelling diffusion is the fine mesh limit of the vertex repelling jump processes involved in the case of grid graphs on the line. More generally, it was shown (\cite{warren98,Aidekon}) that the self-repelling diffusion inverting the Ray-Knight identity on the positive half line can be constructed with the Jacobi flow. 
 
 We want to explore the relationship between the  repelling jump processes in (Q1)-(Q3)  and the self-repelling diffusions in \cite{Aidekon,1998BROWNIAN,warren98}. Our last question is 
 
 (Q4) Are the fine mesh limits of the repelling jump processes involved in (Q1)-(Q3) the self-repelling diffusions? 

In this paper, we will focus on nearest-neighbour CTMCs on a tree $\T$ and give a complete answer to the above  questions (Q1)-(Q4).  It is shown that the percolation Ray-Knight identity has a  simple form in this case, see Theorem \ref{percolationrayknight}.  We construct two kinds of  repelling jump processes, namely the vertex-repelling jump process and the percolation-vertex repelling jump process that invert the Ray-Knight identity and percolation Ray-Knight identity related to loop soup respectively, and show that the fine mesh limits of these  repelling jump processes are self-repelling diffusions  involved in the inversion of the Ray-Knight identity on the metric graph associated to $\T$. Besides, the inversion processes in the case of general graphs have been constructed in  another paper of ours in preparation. The inversion processes on general graphs involve non-local jump rates. So we restrict our discussion to the case of trees here.

The main feature of this paper is the intuitive way of constructing vertex repelling jump processes, which is rather  different from \cite{2016Inverting,lupu2019inverting}. It is enlightened by the recovery of the loop soup with intensity $1/2$ given its local time field (see \cite[\S2.5]{Werner2020} and \cite[Proposition 7]{Werner2016}), where Werner involves the crossings of loop soup that greatly simplify the recovery. In our case, the introduction of crossings translates the problem into a `discrete-time version' of inverting Ray-Knight identity, which can be 
stated as recovering the path of a discrete-time Markov chain conditioned on the number of crossings over each edge, see Proposition \ref{chainsjiaqiang}. This inversion has a surprisingly nice description, which can be seen as a `reversed' oriented version of the edge-reinforced random walk.
   
The paper is organized as follows.  In \S\ref{main666},  we introduce the Ray-Knight identity and percolation Ray-Knight identity related to loop soup and give the main results of the paper. In \S\ref{invertingidsection}-\ref{invertpercolation}, the vertex repelling jump process and the percolation-vertex repelling jump process are shown to invert the Ray-Knight identity and the percolation Ray-Knight identity respectively. In \S\ref{convergencesection}, we verify that the mesh limits of   repelling jump processes are the self-repelling diffusions. In Appendix \ref{selfinteracting}, we give the rigorous definition and basic properties of a class of processes, called processes with terminated jump rates, covering the repelling jump processes. 

\section{Statements of main results}\label{main666}
In this section, we  first recall the Ray-Knight identity. Then we introduce   a new Ray-Knight identity that we call percolation Ray-Knight identity. Finally, we present our results concerning the inversion of these identities and the fine mesh limit of the inversion processes.

\subsection*{Notations}
We will use the following notations throughout the paper. $\N=\{0,1,2,\cdots\}$, $\r^+=[0,\infty)$. For any stochastic process $R$ on some state space $\mathcal{S}$, for $t,u\ge 0$, $x\in \mathcal{S}$ and a specified point $x_0\in \mathcal{S}$, we denote by
\begin{itemize}
    \item  $L^{R}(t,x)$ the local time of $R$. When $\mathcal{S}$ is discrete, $L^R(t,x):=\int_0^t 1_{\{R_s=x\}}\d s$;
    \item $u\mapsto\tau_u^R$ the right-continuous inverse of $t\mapsto L^R(t,x_0)$;
    \item $T^R$ the lifetime of $R$;
    \item $H_x^R:=\inf\{t>0:R_t=x\}$ the hitting time of $x$;
    \item (When $\mathcal{S}$ is discrete) $J_i^R=J_i(R)$ the $i$-th jump time of $R$;
    \item $R_{[0,t]}:=(R_s:0\le s\le t)$ the path of $R$ up to time $t$.
\end{itemize}
The superscripts in above notations are omitted when $R=X$, the CTMC to be introduced immediately.

\subsection{Ray-Knight identity related to loop soup}\label{firstsubsection}
Consider a tree $\T$, i.e. a finite or countable connected  graph without cycle, with root $x_0$. Denote by $V$ its set of vertices, by $E$, resp. $\vec{E}$, its set of undirected, resp. directed, edges. Assume that any vertex $x\in V$ has finite degree. We  write $x<y$ if $x$ is an ancestor of $y$ and $x\sim y$ if $x$ and $y$ are neighbours. Denote by $\mathfrak{p}(x)$ the parent of $x$. For $x\sim y$, we simply write $xy:=(x,y)$ for a directed edge. The tree is endowed with a killing measure $\left(k_x\right)_{x\in V}$ on $V$ and conductances $\left(C_
{xy}\right)_{xy\in \vec{E}}$ on $\vec{E}$. We do not assume the symmetry of the conductances at the moment. Write $C^*_{xy}=C^*_{yx}=\sqrt{C_{xy}C_{yx}}$ for $x\sim y$.

Consider the CTMC $X=(X_t)_{0\le t< T^{X}}$ on $V$ which being at $x$, jumps to $y$ with rate $C_{xy}$ and is killed with rate $k_x$. $T^{X}$ is the time when the process is killed or explodes. Let  $\cL$ be the unrooted oriented loop soup with some fixed intensity $\alpha>0$ associated to $X$. (See for example \cite{lejan17} for the precise definition.)


Denote by $L_{\cdot}(\cL)$ the occupation time field of $\cL$, i.e. for all $x\in V$, $L_x(\cL)$ is the sum of the local time at $x$ of each loop in $\cL$. It is well-known that when $X$ is transient, $L_x(\cL)$ follows a Gamma($\alpha, G(x,x)^{-1}$) distribution\footnote{The density of  Gamma($a,b$) distribution at $x$ is $1_{\{x>0\}}\frac{b^a}{\Gamma(a)}x^{a-1}e^{-b x}$.}, where $G$ is the Green function of $X$; when $X$ is recurrent, $L_x(\cL)=\infty$ for all $x\in V$ a.s..  We first suppose $X$ is transient, which ensures that the conditional distribution of $\cL$ given $L_{x_0}(\cL)$ exists. For $u\ge 0$, let $\cL^{(u)}$ have the law of $\cL$ given $L_{x_0}(\cL)=u$. Without particular mention, we always assume that $X$ starts from $x_0$. 
The next proposition (see \cite[Proposition 3.7]{lupu2019inverting} or \cite[Proposition 5.3]{chang2016markov}) connects the path of $X$ with the loops in $\cL^{(u)}$ that visit $x_0$. 
\begin{proposition}\label{PDpar}
For any $u>0$, consider the path  $\left(X_t\right)_{0\le t\le\tau_u}$ conditioned on $\tau_u<T^{X}$. Let $D=\left(d_1, d_2\cdots\right)$ be a Poisson-Dirichlet partition with parameter $\alpha$, independent of $X$.
Set $s_n:=u\cdot\sum_{k=1}^{n} d_k$.
Then the family of unrooted loops
\[\left\{\pi\left(\big(X_{\tau_{s_{j-1}}+t}\big)_{0\le t\le \tau_{s_j}-\tau_{s_{j-1}}}\right):j\ge1\right\}\]
is distributed as all the loops in $\cL^{(u)}$ that visit $x_0$, where $\pi$ is the quotient map that maps a rooted loop to its corresponding unrooted loop. 
\end{proposition}

By Proposition \ref{PDpar}, we can take $\allowbreak\cL^{(0)}=\{\gamma\in\cL:\gamma \text{ does not}\linebreak[0]\text{ visit }x_0\}$ and $\cL^{(u)}$ to be the collection of loops in $\cL^{(0)}$ and  loops derived by partitioning $X_{[0,\tau_u]}$, where $X$ and $\cL^{(0)}$ are required to be independent. This special choice of $\left(\cL^{(u)}:u\ge 0\right)$ provides a continuous version of the conditional distribution. 
 So  we work on this version from now on. Note that the above definition also makes sense when $\alpha=0$. In this case, $\cL^{(0)}=\emptyset$ and $\cL^{(u)}$ consists of a single loop $\pi\left(X_{[0,\tau_u]}\right)$ (Poisson-Dirichlet partition with parameter $0$ is considered as the trivial partition $D=(1,0,0,\cdots)$). So we also allow $\alpha=0$ from now on. The generalized second Ray-Knight theorem related to loop soup reads as follows, which is direct from the above definition.

\begin{proposition}[Ray-Knight identity]\label{rayknighttransient}
Let $\cL^{(0)}$ and $X$ be independent. Then for $u>0$, conditionally on $\tau_u<T^{X}$,
\[\left(L_x(\cL^{(0)})+L(\tau_u,x)\right)_{x\in V} \text{ has the same law as } \left(L_x(\cL^{(u)})\right)_{x\in V}.\]
\end{proposition}

\subsection{Percolation Ray-Knight identity related to loop soup}\label{percoray}
In this part, we assume the symmetry of the conductances (i.e. $C_{xy}=C_{yx}$ for any $x\sim y$) and that $0<\alpha<1$. An element $O$ in $\{0,1\}^E$ is also called a configuration on $E$. When $O(e)=1$, the edge $e$ is thought of as being open. $O$ is also used to stand for  the set of open edges induced by $O$ without particular
mention. A percolation on $E$ refers to a random configuration on $E$.


\begin{definition}\label{alphapercolation}
For $u\ge 0$, let $\O^{(u)}$ be the  percolation on $E$ such that, conditionally on $L_\cdot (\cL^{(u)})=\ell$: 
\begin{longlist}
\item  edges are open independently; 
\item the edge $\{x,y\}$ is open with probability
\begin{align}\label{zez}
\dfrac{I_{1-\alpha}\left(2C_{xy}\sqrt{\ell_x\ell_y}\right)}{I_{\alpha-1}\left(2C_{xy}\sqrt{\ell_x\ell_y}\right)},
\end{align}
where $I_\nu$ is the modified Bessel function: for $\nu\ge -1$ and $z\ge 0$,
\[I_\nu(z)=(z/2)^\nu\sum_{n=0}^\infty\dfrac{(z/2)^{2n}}{n!~\Gamma(n+\nu+1)}.\]
\end{longlist}
\end{definition}




 Consider the loop soup $\cL^{(0)}$ and the percolation $\O^{(0)}$. For $t\ge 0$, define the aggregated local times
\begin{align}\label{lambdatdef}
\phi_t(x):=L_x(\cL^{(0)})+L(t,x),~x\in V. 
\end{align}

The process $\cX:=(X_t,\O_t)_{0\le t\le \tau_u}$ is defined as follows: $\left(X_0,\O_0\right):=\left(x_0,\O^{(0)}\right)$. Conditionally on $L_\cdot(\cL^{(0)})$ and $(\cX_s:0\le s\le t)$, if $X_t=x$ and $y\sim x$, then 
\begin{itemize}
\item $X_t$ jumps to its neighbour $y$ with rate $C_{xy}$ and $\O_t(\{x,y\})$ is then set to $1$ (if it was not already).
\item In case $\O_t(\{x,y\})=0$, $\O_t(\{x,y\})$ is set to $1$ without $X_t$ jumping with rate
\begin{align}\label{o1}
C_{xy}\sqrt{\dfrac{\phi_t(y)}{\phi_t(x)}}\cdot\dfrac{K_{\alpha}\big( C_{xy}\sqrt{\phi_t(x)\phi_t(y)}\big)}{K_{1-\alpha}\big(C_{xy}\sqrt{\phi_t(x)\phi_t(y)}\big)},
\end{align}
where 
$K_\nu(z)=\dfrac{1}{2}\Gamma(\nu)\Gamma(1-\nu)\left(I_{-\nu}(z)-I_\nu(z)\right)$.
\end{itemize}

\begin{theorem}[Percolation Ray-Knight identity]\label{percolationrayknight}
With the notations above, conditionally on $\tau_u<T^{X}$, $\allowbreak(\phi_{\tau_u},\linebreak[0]\O_{\tau_u})$ has the same law as $(L_\cdot(\cL^{(u)}), \O^{(u)})$.
\end{theorem}

Theorem \ref{percolationrayknight} will be proved in \S\ref{invertpercolation}. The process $(\O_t)_{0\le t\le \tau_u}$ has a natural interpretation in terms of loop soup on a metric graph. Specifically, let $\tG$ be the metric graph associated to $\T$, where  edges are considered  as intervals, so that one can construct a Brownian motion $B$ which moves continuously on $\tG$, and whose print on the vertices is distributed as $X$ (see \S\ref{invertpercolation} for details). Let $\wL$ be the unrooted oriented loop soup with   intensity $\alpha>0$ associated to  $B$. 
Starting with the loop soup $\wL$ with boundary condition $0$ at $x_0$ and an independent Brownian motion $B$ starting at $x_0$, one can consider the field $(\widetilde{\phi}_t(x),\, x\in \tG)$ which is the aggregated local time at $x$ of the loop soup $\wL$ and the Brownian motion $B$ up to time $t$. Then one can construct $\O_t$ as the configuration where an edge $e$ is open if the field $\widetilde{\phi}_t$ does not have any zero on the edge $e$. See Proposition~\ref{cxcx}.

\begin{remark}
Since the laws involved in the Ray-Knight and percolation Ray-Knight do not depend on $k_{x_0}$, we can generalize the above results to the case where $X$ is recurrent. 
\end{remark}

\subsection{Inversion of Ray-Knight identities}\label{qwer666}  

 To ease the presentation, write  $\phi^{(u)}$ for $L_\cdot(\cL^{(u)})$ from now on. Theorem \ref{percolationrayknight} allows us to identify $(\phi^{(u)},\O^{(u)})$ with $\allowbreak(\lambda_{\tau_u},\linebreak[0]\O_{\tau_u})$, which we will do. 
\begin{definition}
We call the triple 
$$
\left(\phi^{(0)}, X_{[0,\tau_u]},\phi^{(u)}\right)
$$

\noindent a Ray-Knight triple (with parameter $\alpha$) associated to $X$. Similarly, recalling the notation $\cX_t=(X_t,\O_t)$, the triple 
$$
\left((\phi^{(0)}, \O^{(0)}), \cX_{[0,\tau_u]},(\phi^{(u)}, \O^{(u)}) \right)
$$

\noindent will be called a percolation Ray-Knight triple. 
\end{definition}

Inverting the Ray-Knight, resp. percolation Ray-Knight identity, is to deduce the conditional law of $X_{[0,\tau_u]}$, resp. $\cX_{[0,\tau_u]}$, given $\phi^{(u)}$, resp. $(\phi^{(u)}, \O^{(u)})$.


We introduce an adjacency relation on $V\times \{0,1\}^E$: For $(x,O_1), (y,O_2)\in V\times \{0,1\}^E$, $(x,O_1)$ and $(y,O_2)$ are neighboured if they satisfy either one of the followings: (1) $O_1=O_2$ and $x\sim y$; (2) $O_1$ and $O_2$ differ by exactly one edge $e$ and $x,y\in e$. This defines a graph $\T^\prime$ with finite degree and $\cX_t$ is a nearest-neighbour jump process on $\T^\prime$.

The inversion of Ray-Knight identities is expressed in terms of  several processes defined by jump rates. Readers are referred to Appendix \ref{selfinteracting} for the rigorous definition and basic properties of such processes. All the continuous-time processes defined below are assumed to be right-continuous, minimal, nearest-neighbour jump processes with a finite or infinite lifetime. The collection of all such sample paths on $\T$ (resp. $\T^\prime$)  is denoted by $\Omega$ (resp. $\Omega^\prime$).


Given $\lambda\in (\r^+)^V$, set for $x,y\in V$ with $x\sim y$ and $\omega\in \Omega$ with $T^\omega>t$,
\begin{align}\label{ratesymbol}
\begin{split}
   &\Lambda_t(x)=\Lambda_t(\lambda,\omega)(x)=\Lambda_t(\lambda,\omega_{[0,t]})(x):=\lambda(x)-\int_0^t 1_{\{\omega_s=x\}}\d s,\\
   &\varphi_t(xy)= \varphi_t(\lambda,\omega)(xy)=\varphi_t(\lambda,\omega_{[0,t]})(xy):=2C^*_{xy}\sqrt{\Lambda_t(x,\omega) \Lambda_t(y,\omega)}.\\
   \end{split}
\end{align}
Intuitively, $\lambda$ is viewed as the initial local time field and $\Lambda_t$ stands for the remaining local time field while running the process $\omega$ until time $t$.  Although these quantities depend on $\omega$, we will systematically drop the notation $\omega$ for sake of concision whenever the process  is clear from the context. 
 

\subsubsection{Inversion of Ray-Knight identity}\label{s:inversionrk}
Keep in mind that most of the following definitions have a parameter $\alpha\ge 0$, which is always omitted in notations for simplicity.

Set
\begin{align*}
    \mathfrak{R}=\begin{cases}
    \big\{\lambda\in (\r^+)^V:\lambda(x)>0\ \forall x\in V\big\},&\text{ if }\alpha>0;\\
    \big\{\lambda\in (\r^+)^V:\lambda(x_0)>0,\,\text{supp}(\lambda)\text{ is connected and finite}\big\},&\text{ if }\alpha=0.
\end{cases}
\end{align*}
Note that with probability $1$, $\phi^{(u)}\in \mathfrak{R}$. We will take $\mathfrak{R}$ as the range of $\phi^{(u)}$ and consider only the law of $X_{[0,\tau_u]}$ given $\phi^{(u)}=\lambda\in\mathfrak{R}$. \\

Now define the \textit{vertex repelling jump process} we are interested in. Given $\lambda\in\mathfrak{R}$, its distribution $\p^\lambda_{x_0}$ on $\Omega$ verifies that the process $\omega=(\omega_t,0\le t< T^\omega)$ starts at $\omega_0=x_0$, behaves such that
\begin{itemize}
\item conditionally on $t<T^\omega$ and $\left(\omega_s:0\le s\le t\right)$ with $\omega_t=x$, it jumps to a neighbour $y$ of $x$ with rate $r^\lambda_t\big(x,y,\omega_{[0,t]}\big)$;
\item (resurrect mechanism). every time $\lim\limits_{s\rightarrow t-}\Lambda_s(\omega_s)=0$ and $\omega_{t-}\neq x_0$, it jumps to $\mathfrak{p}(x)$ at time $t$,
\end{itemize}
and stops at time $T^\omega=T^\lambda(\omega)$, 
where for $x,y\in V$ with $x\sim y$ and $\omega\in \Omega$ with $T^\omega>t$,
\begin{small}
\begin{align}\label{xxrate}
\begin{split}
    r^\lambda_t(x,y,\omega_{[0,t]})&:=\left\{\begin{aligned}
    &C^*_{xy}\sqrt{\dfrac{\Lambda_t(y)}{\Lambda_t(x)}}\cdot\dfrac{I_{\alpha-1}\left(\varphi_t(xy)\right)}{I_{\alpha}\left(\varphi_t(xy)\right)},&\text{ if $y=\mathfrak{p}(x)$},\\
    &C^*_{xy}\sqrt{\dfrac{\Lambda_t(y)}{\Lambda_t(x)}}\cdot\dfrac{I_{\alpha}\left(\varphi_t(xy)\right)}{I_{\alpha-1}\left(\varphi_t(xy)\right)},&\text{  if $x=\mathfrak{p}(y)$},\\
    \end{aligned}\right.
    \end{split}
\end{align}
\end{small}and $T^\lambda=T^\lambda_0\wedge T^\lambda_\infty$ with
\begin{align*}
    T^\lambda_0(\omega)&:=\sup\big\{t\ge 0:\Lambda_t(x_0)>0\big\};\\
    T^\lambda_\infty(\omega)&:=\sup\big\{t\ge 0:\omega_{[0,t]} \text{ has finitely many jumps}\big\}.
\end{align*}
Here $T^\lambda_0$ represents the time when the local time at $x_0$ is exhausted. The process can be roughly described as follows: the total local time available at each vertex is given at the beginning. As the process runs, it eats the local time. The jump rates are given in terms of the remaining local time. It finally stops whenever the available local time at $x_0$ is used up or an explosion occurs. 

\begin{remark}\label{Xdifferent}
It holds that $I_1=I_{-1}$ and as $z\downarrow 0$,
\begin{align*}
    I_\nu(z)\sim\begin{cases}
    \Gamma(\nu+1)^{-1}(\frac{1}{2}z)^\nu,&\text{ if }\nu>-1;\\
    \frac{1}{2}z,&\text{ if }\nu=-1.\end{cases}
\end{align*}
Hence,
\begin{align}\label{zgoto}
 \text{for $\alpha>0$},~I_{\alpha-1}(z)/I_\alpha(z)\sim \alpha z^{-1};\quad  I_{-1}(z)/I_0(z)\sim z.
\end{align}
The different behaviors in \eqref{zgoto} for $\alpha>0$ and $\alpha=0$ indicate different behaviors of the vertex repelling jump process in the two cases. Intuitively speaking, when $\alpha>0$, as the process tends to stay at some $x\neq x_0$, the jump rate to $\mathfrak{p}(x)$ goes to infinity. So it actually does not need the resurrect mechanism in this case, and there is still  some positive local time left at any vertex other than $x_0$ at the end.
While in the case of $\alpha=0$, the process exhibits a different picture. Contrary to the case of $\alpha>0$, the jump rate to the children goes to infinity when $\alpha=0$. So intuitively the process is `pushed' to the boundary of $\text{supp}(\lambda)$ and exhausts the available local time at one of the boundaries. To guide the process back to $x_0$, we decide to resurrect it by letting it jump to the parent of the vertex. In this way, the process finally ends up exhausting all the available local time at each vertex.
\end{remark}

By Remark \ref{Xdifferent}, we can see that the vertex repelling jump process is in line with our intuition about the inversion of Ray-Knight identity. Since in the case of $\alpha>0$, the given time field includes the external time field $\phi^{(0)}$, the inversion process will only use part of the local time at any vertex other than $x_0$ and end up exhausting the local time at $x_0$. While in the case of $\alpha=0$, the given time field is exactly the local time field of the CTMC itself, the inversion process will certainly use up the local time at each vertex before  finally stopping at $x_0$.

\begin{theorem}\label{invertfinite}
Suppose $\left(\phi^{(0)}, X_{[0,\tau_u]}, \phi^{(u)}\right)$ is a Ray-Knight triple associated to $X$. For any $\lambda\in\mathfrak{R}$, the conditional distribution of $X_{[0,\tau_u]}$ given $\tau_u<T^{X}$ and $\phi^{(u)}=\lambda$ is $\p^\lambda_{x_0}$. 
\end{theorem}

\subsubsection{Inversion of percolation Ray-Knight identity}\label{inperco}
Assume the symmetry of the conductances and that $0<\alpha<1$. Our goal is to introduce the \textit{percolation-vertex repelling jump process} that
inverts the percolation Ray-Knight identity. 
Given $\lambda\in \mathfrak{R}$ and a configuration $O$ on $E$, the distribution $\p^{\lambda}_{(x_0,O)}$ on $\Omega^\prime$ verifies that the process $\omega=(\omega_t,0\le t<T^\omega)$ (the first coordinate being a jump process on $V$ and the second coordinate a percolation on $V$) starts at $(x_0,O)$, moves such that conditionally on $t<T^\omega$ and $(\omega_s:0\le s\le t)$ with $\omega_t=(x,O_1)$, it jumps from $(x,O_1)$ to $(y,O_2)$ with rate
\begin{small}
\begin{align}\label{cxrate}
\begin{split}
    \left\{\begin{aligned}
    &1_{\{O_1(\{x,y\})=1\}}C_{xy}\sqrt{\dfrac{\Lambda_t(y)}{\Lambda_t(x)}}\cdot\dfrac{I_{1-\alpha}\left(\varphi_t(xy)\right)}{I_{\alpha}\left(\varphi_t(xy)\right)},&\text{ if $y=\mathfrak{p}(x), O_1=O_2$};\\
    &1_{\{O_1(\{x,y\})=1\}}C_{xy}\sqrt{\dfrac{\Lambda_t(y)}{\Lambda_t(x)}}\cdot\dfrac{I_{\alpha}\left(\varphi_t(xy)\right)}{I_{1-\alpha}\left(\varphi_t(xy)\right)},&\text{ if $x=\mathfrak{p}(y),  O_1=O_2$};\\
     &C_{xy}\sqrt{\dfrac{\Lambda_t(y)}{\Lambda_t(x)}}\cdot\dfrac{I_{\alpha-1}\left(\varphi_t(xy)\right)-I_{1-\alpha}\left(\varphi_t(xy)\right)}{I_{\alpha}\left(\varphi_t(xy)\right)},&\text{ if $y=\mathfrak{p}(x), O_2=O_1\setminus\{x,y\}$};\\
 &C_{xz}\sqrt{\dfrac{\Lambda_t(z)}{\Lambda_t(x)}}\cdot\dfrac{I_{-\alpha}\left(\varphi_t(xz)\right)-I_{\alpha}\left(\varphi_t(xz)\right)}{I_{1-\alpha}\left(\varphi_t(xz)\right)},&\text{ if }x=y, O_2=O_1\setminus\{x,z\}\atop \text{ for $z$ with } x=\mathfrak{p}(z),\\
     \end{aligned}\right.
    \end{split}
\end{align}
\end{small}

\noindent and stops at time $T^\omega=T^{\lambda,O}(\omega)$ when the process explodes or uses up the local time at $x_0$. Here for $\omega=(\omega^1,\omega^2)\in \Omega^\prime$,  $\Lambda_t(x,\omega)$, $\varphi_t(xy,\omega)$ are defined as $\Lambda_t(x,\omega^1)$, $\allowbreak\varphi_t(xy,\linebreak[0]\omega^1)$ respectively. 

\begin{theorem}\label{invertfinitecluster}
Suppose $((\phi^{(0)}, \O^{(0)}), \cX_{[0,\tau_u]},(\phi^{(u)}, \O^{(u)}))$ is a percolation Ray-Knight triple associated to $X$.
For any $\lambda\in \mathfrak{R}$ and configuration $O$ on $E$, the conditional distribution of $\left(\cX_{\tau_u-t}\right)_{0\le t\le \tau_u}$  given $\left(\phi^{(u)}, \O^{(u)}\right)=(\lambda,O)$ and $\tau_u<T^X$ is $\p^{\lambda}_{(x_0,O)}$.
\end{theorem}

\subsection{Mesh limit of vertex repelling jump processes}\label{scalelimit}
We only tackle the case of simple random walks on dyadic grids, which can be easily generalized to general CTMCs on trees. Let $B$ be a reflected Brownian motion on $[0,\infty)$. View $0$ as the `root' and let $(\widetilde{\phi}^{(0)},B_{[0,\tau^B_u]},\widetilde{\phi}^{(u)})$ be a Ray-Knight triple associated to $B$ (defined in a similar way to that for CTMCs). The conditional law of $B_{[0,\tau^B_u]}$ given $\widetilde{\phi}^{(u)}=\lambda$ is the self-repelling diffusion $W^\lambda$, which can be constructed with the burglar process. See \S\ref{convergencesection} for details.

Denote $\N_k:=2^{-k}\N$. Consider $\T_k=(\N_k,E_k)$, where $E_k:=\big\{\{x,y\}:x,y\in \N_k,\ \abs{x-y}=2^{-k}\big\}$, endowed with conductances $C_e^k=2^{k-1}$ on each edge and no killing. The induced CTMC $X^{(k)}$ is the print of $B$ on $\N_k$. Let $\big(\phi^{(0),k},X^{(k)}_{[0,\tau^{X^{(k)}}_u]},\phi^{(u),k}\big)$ be a Ray-Knight triple associated to $X^{(k)}$. It holds that
\[\left(X^{(k)}_{2^kt}, L^{X^{(k)}}(2^kt,x),\phi^{(0),k}(x)\right)_{t\ge 0,x\in \r^+}\overset{d}{\rightarrow}\left(B_t,L^B(t,x),\widetilde{\phi}^{(0)}(x)\right)_{t\ge 0,x\in \r^+},\]
where $L^{X^{(k)}}(2^kt,\cdot)$ and $\phi^{(0),k}(\cdot)$ are considered to be linearly interpolated outside $\N_k$.

In view of this, for any non-negative, continuous function $\lambda$ on $\r^+$ with some additional conditions,
we naturally consider the vertex repelling jump process $X^{\lambda,(k)}$, which has the conditional law of $\left(X^{(k)}_{2^kt}:0\le t\le 2^{-k}\tau^{X^{(k)}}_u\right)$ given $\phi^{(u),k}=\lambda\vert _{\N_k}$. The jump rates of the process are given in \eqref{xrate}.

\begin{theorem}\label{convergence}
For $\lambda\in\widetilde{\mathfrak{R}}$ (defined in \S\ref{convergencesection}), the family of vertex repelling jump processes $X^{\lambda,(k)}$ converges weakly as $k\rightarrow\infty$ to the self-repelling diffusion $W^\lambda$ for the uniform topology, where the processes are assumed to stay at $0$ after the lifetime. 
\end{theorem}

\section{Inverting Ray-Knight identity}\label{invertingidsection}

In this section, we will obtain the inversion of Ray-Knight identity. The main idea is to first introduce the information on crossings and explore the law of CTMC conditioned on both local times and crossings. Then by `averaging over crossings', we get the representation of the inversion as the vertex repelling jump process shown in Theorem \ref{invertfinite}.

To begin with,  observe that it suffices to only consider  the case when $X$ is recurrent. In fact, for $x\in V$, we set $h(x)=\p^x(H_{x_0}<T^X)$, where $\p^x$ is the law of $X$ starting from $x$. Let $Y$ be the CTMC on $V$ starting from $x_0$ induced by conductances $C^h_{xy}=\frac{h(y)}{h(x)}C_{xy}$ and no killing. Then by the standard argument, we can show that $Y$ is recurrent and $Y_{[0,\tau^Y_u]}$ has the law of $X_{[0,\tau_u]}$ conditioned on $\tau_u<T^X$. Note that $Y$ can also be obtained by removing the killing rate at $x_0$ from the $h$-transform of $X$ (the latter process is killed at $x_0$ with rate $C_{x_0}-C^h_{x_0}$, where $C_x=k_x+\sum_{y:y\sim x}C_{xy}$ and $C^h_x=\sum_{y:y\sim x}C^h_{xy}$). Combine the two facts: (1) the law of the loop soup is invariant under the $h$-transform (Cf. \cite[Proposition 3.2]{chang2016markov}); (2) the law of the Ray-Knight triple does not depend on the killing rate at $x_0$. We have the Ray-Knight triple associated to $X$ has the same law as that associated to $Y$. Then it is easy to deduce Theorem \ref{invertfinite} in the transient case from that in the recurrent case.

Throughout this section, we will assume $X$ is recurrent.

\subsection{The representation of the inversion as a vertex-edge repelling process}\label{representation1}
\begin{definition}\label{rnet}
An element  $n=\left(n(xy)\right)_{xy\in \vec{E}}\in \mathbb{N}^{\vec{E}}$ is called a network (on $\T$). For any network $n$, set
\begin{align*}
\begin{split}
    \widecheck{n}(xy)= n(xy)+(\alpha-1)\cdot 1_{\{y=\mathfrak{p}(x)\}},
    \end{split}
\end{align*}
and for $x\in V$, $n(x):=\sum_{y:y\sim x} n(xy)$ and $\widecheck{n}(x):=\sum_{y:y\sim x} \widecheck{n}(xy)$. If $n(xy)=n(yx)$ for any $\{x,y\}$ in $E$, we say that the network $n$ is sourceless, denoted by $\partial n= \emptyset$. Given $\lambda\in \mathfrak{R}$, we call $\mathcal{N}$ a sourceless $\alpha$-random network  associated to $\lambda$ if $\mathcal{N}$ is a sourceless network, and $\left(\mathcal{N}(xy): x=\mathfrak{p}(y)\right)$ are independent with $\mathcal{N}(xy)$ following the Bessel~$(\alpha-1$, $2C^*_{xy}\sqrt{\lambda(x)\lambda(y)})$ distribution\footnote{For $\nu\ge -1$ and $z>0$, the Bessel$(\nu,z)$ distribution is a distribution on $\N$ given by:
\begin{align}\label{bessel1}
    b_{\nu,z}(n)=I_\nu(z)^{-1}\dfrac{(z/2)^{2n+\nu}}{n!~\Gamma(n+\nu+1)},~n\in \N.
\end{align} 

Bessel$(\nu,0)$ distribution is defined to be the Dirac measure at $0$.}.

More generally, let $\mathfrak{p}(x_0,x)$ be the unique self-avoiding path from $x_0$ to $x$, also seen as a collection of unoriented edges. For $i\in V\backslash\{x_0\}$, we say that $n$ has sources $(x_0,i)$, denoted by $\partial n=(x_0,i)$, if for any $xy\in \vec{E}$ with $x=\mathfrak{p}(y)$,
\[\left\{\begin{array}{ll}
   n(xy)=n(yx)-1,&\text{ if }\{x,y\}\in \mathfrak{p}(x_0,i);\\
   n(xy)=n(yx),&\text{ if }\{x,y\}\notin \mathfrak{p}(x_0,i).
\end{array}\right.\]
Given $\lambda\in \mathfrak{R}$ and $i\in V\backslash\{x_0\}$, we call $\mathcal{N}$ an $\alpha$-random network  with sources $(x_0,i)$ associated to $\lambda$ if $\mathcal{N}$ is a network with sources $(x_0,i)$ and $\left(\mathcal{N}(xy): x=\mathfrak{p}(y)\right)$ are independent with $\mathcal{N}(xy)$ following the Bessel~$(\alpha$, $2C^*_{xy}\sqrt{\lambda(x)\lambda(y)})$ distribution if $x<i$, and the Bessel~$(\alpha-1$, $2C^*_{xy}\sqrt{\lambda(x)\lambda(y)})$ distribution otherwise.
\end{definition} 

\noindent {\it Remark}. We will sometimes use the convention that a network with sources $(x_0,x_0)$ is a sourceless network. 

\

Every loop configuration $\mathscr{L}$ (i.e. a collection of unrooted, oriented loops) induces a network $\theta({\mathscr{L}})$: for $x\sim y$,
\[\theta({\mathscr{L}})(xy):=\#\text{ crossings from $x$ to $y$ by the loops in $\mathscr{L}$}.\]
Due to the tree structure, it holds that $\theta({\mathscr{L}})(xy)=\theta({\mathscr{L}})(yx)$ for any $xy\in \vec{E}$, i.e. $\theta({\mathscr{L}})$ is sourceless. 

Let $\left(\phi^{(0)},X_{[0,\tau_u]},\phi^{(u)}\right)$ be the Ray-Knight triple associated to $X$, where $\phi^{(0)}$ is the local time field of a loop soup $\cL^{(0)}$ independent of $X$. The path $X_{[0,\tau_u]}$ is also viewed as a loop configuration consisting of a single loop. Let $\cN^{(u)}:=\theta({\cL^{(0)}})+\theta({X_{[0,\tau_u]}})$. 
We have the following result, the proof of which is contained in \S\ref{proofofinitialcrossings}.  \cite[Theorem 3.1]{knight1998upcrossing} provides another proof for the case of $\alpha=0$.

\begin{proposition}\label{initialcrossing}
For $\lambda\in \mathfrak{R}$, conditionally on $\phi^{(u)}=\lambda$, $\cN^{(u)}$ is a sourceless $\alpha$-random network  associated to $\lambda$.
\end{proposition}

With this proposition, for the recovery of $X_{[0,\tau_u]}$, it suffices to derive the law of $X_{[0,\tau_u]}$ given $\phi^{(u)}$ and $\cN^{(u)}$. Set
\begin{align*}
    \mathfrak{N}:=\begin{cases}
    \big\{n\in \N^{\vec{E}}:\partial n=\emptyset \big\},&\text{if }\alpha>0;\\
    \big\{n\in \N^{\vec{E}}:\partial n=\emptyset,\,\text{supp}(n)\text{ is connected and finite}\big\},&\text{if }\alpha=0.
    \end{cases}
\end{align*}
Given $n\in\mathfrak{N}$, for $\omega\in \Omega$ and $x,y\in V$ with $x\sim y$, set
\begin{align}\label{def:thetat}
\begin{aligned}
    \Theta_t(xy)&=\Theta_t(n,\omega)(xy)=\Theta_t(n,\omega_{[0,t]})(xy)\\
    &:=n(xy)-\#\left\{0<s\le t:\omega_{s-}=x,\omega_s=y\right\},    
\end{aligned}
\end{align}
which represents the remaining crossings while running the process $\omega$ until time $t$ with the initial crossings  $n$. Recall notations introduced at the beginning of \S\ref{s:inversionrk}. Given $\lambda\in \mathfrak{R}$ and $n\in\mathfrak{N}$, the vertex-edge repelling jump process $X^{\lambda,n}$ is defined to be a process starting from $x_0$,  behaves such that
\begin{longlist}
     \item conditionally on $t<T^{X^{\lambda,n}}$ and $\left(X^{\lambda,n}_s:0\le s\le t\right)$ with $X^{\lambda,n}_t=x$, it jumps to a neighbour $y$ of $x$ with rate $r^{\lambda,n}_t\big(x,y,X^{\lambda,n}_{[0,t]}\big)$;
    \item every time $\lim\limits_{s\rightarrow t-}\Lambda_s(X^{\lambda,n}_s)=0$ and $X^{\lambda,n}_{t-}\neq x_0$, it jumps to $\mathfrak{p}(x)$ at time $t$,   
\end{longlist}
and  stops at time $T^{X^{\lambda,n}}$ when the process exhausts the local time at $x_0$ or explodes. Here for $\omega\in \Omega$ with $T^\omega>t$,
\begin{align}\label{yyrate}
\begin{aligned}
r^{\lambda,n}_t(x,y,\omega):=\frac{\widecheck{\Theta}_t(xy)}{\Lambda_t(x)}.
\end{aligned}
\end{align}
In the above expression, $\Theta_t$ is viewed as a network, and $\widecheck{\Theta}_t$ is defined as before.

Intuitively, for this process, both the local time and crossings available are given at the beginning. The process eats the local time during its stay at vertices and consumes crossings at jumps over edges. 

\begin{theorem}\label{rayknightcrossing}
For any $\lambda\in \mathfrak{R}$ and $n\in\mathfrak{N}$,
    $X^{\lambda,n}$ has the law of $X_{[0,\tau_u]}$ conditioned on $\phi^{(u)}=\lambda$ and $\cN^{(u)}=n$.
\end{theorem}
\subsubsection{Proof of Proposition \ref{initialcrossing}}\label{proofofinitialcrossings}

For a subset $V_0$ of $V$, the print of $X$ on $V_0$ is by definition
\begin{align}\label{print}
    \Big(X_{A^{-1}(t)}:0\le t< A\big(T^X\big)\Big),\text{ where } A(u):=\int_0^u 1_{\{X_s\in V_0\}}\d s.
\end{align}
We can also naturally define the print of a (rooted or unrooted) loop. The print of a loop configuration is the collection of the prints of loops in the configuration.
Recall that $\cL^{(u)}$ consists of the partition of the path of $X_{[0,\tau_u]}$ and an independent loop soup $\cL^{(0)}$. By the excursion theory (resp. basic properties of Poisson random measure), the prints of $X_{[0,\tau_u]}$ (resp. $\cL^{(0)}$) on the different branches\footnote{A branch at $x_0$ is defined as a connected component of the tree when removing the vertex $x_0$, to which we add $x_0$} at $x_0$ are independent. In particular, for any $x\sim x_0$, $\cN^{(u)}(x_0x)$ is independent of the prints of $\cL^{(u)}$ on the branches that do not contain $x$. 
By considering other vertices as the root of $\T$, we can readily obtain that given $\phi^{(u)}=\lambda$, $\left(\cN^{(u)}(yz):y=\mathfrak{p}(z)\right)$ are independent and the conditional law of $\cN^{(u)}(yz)$ depends only on $\lambda(y)$ and $\lambda(z)$.

Now it reduces to considering the law of $\cN^{(u)}(yz)$ conditioned on $\phi^{(u)}(y)=\lambda(y)$ and $\phi^{(u)}(z)=\lambda(z)$. We focus on the case of $yz=x_0x$ only, since it is the same for other edges.  It holds that
\begin{longlist}
    \item $\cN^{(u)}(x_0x)$ has a Poisson distribution with parameter $uC_{x_0,x}$;
    \item $\phi^{(u)}(x)$ equals a sum of $\cN^{(u)}(x_0x)$ i.i.d. exponential random variables with parameter $C_{x,x_0}$;
    \item $L_x(\cL^{(0)})$ follows a \allowbreak Gamma($\alpha$, $C_{x,x_0}$) distribution (a Gamma($0$, $\beta$) r.v. is interpreted as a r.v. identically equal to $0$). In fact, $L_x(\cL^{(0)})$ follows a Gamma($\alpha,G^{\bar{x}_0}(x,x)^{-1}$) distribution, where $G^{\bar{x}_0}$ is the Green function of the process $X$ killed at $x_0$. The recurrence of $X$ implies that $G^{\bar{x}_0}(x,x)=C_{x,x_0}^{-1}$.
    \item The above exponential random variables, $\cN^*(x_0x)$ and $L_x(\cL^{(0)})$ are mutually independent.    
\end{longlist}
   
It follows that the conditional law of $\cN^{(u)}(x_0x)$ is the same as the conditional law of $U$ given
\begin{align*}
    R_*+R_1+\cdots+R_U=\lambda(x),
\end{align*}
when $U$ has Poisson($uC_{x_0,x}$) distribution, $R_*$ has Gamma($\alpha$, $C_{x,x_0}$) distribution, $R_1,R_2,\cdots$ have Exponential($C_{x,x_0}$) distribution and $U,R_*,R_1,R_2,\cdots$ are mutually independent. By directly writing the density of $R_*+R_1+\cdots+R_U$ and then conditioning on the sum being $\lambda(x)$, we can readily obtain that the conditional distribution is Bessel~($\alpha-1,2C^*_{x_0x}\sqrt{u\lambda(x)}$) (see also \cite[\S2.7]{feller1957introduction}). We have thus proved the proposition.
\subsubsection{Proof of Theorem \ref{rayknightcrossing}}

The recovery of $X_{[0,\tau_u]}$ given $\phi^{(u)}$ and $\cN^{(u)}$ is carried out by the following two steps: (1)  reconstruct the jump chain of $X_{[0,\tau_u]}$ conditioned on $\phi^{(u)}$ and $\cN^{(u)}$; (2) assign the holding times before every jump conditioned on $\phi^{(u)}$, $\cN^{(u)}$ and the jump chain. We shall prove the process recovered by the above two steps is exactly the vertex-edge repelling jump process.

For step (1), it is easy to see that the conditional law of the jump chain actually depends only on $\cN^{(u)}$.
Moreover, let $\bar{X}$ be distributed as the jump chain of $X$.
For $k\in \z^+$, let $\bar{\tau}_k:=\inf\{l>0:\#\{1\le j\le l:\bar{X}_j=x_0\}=k\}$. Fixing a sourceless network $n$, set $m=n(x_0)$. We consider $\bar{X}$ up to $\bar{\tau}_m$. Let $\bar{\cL}^{(0)}$ have the law of the discrete loop soup induced by $\cL^{(0)}$, and be independent of $\bar{X}$. Set $\bar{\cN}:=\theta({\bar{X}_{[0,\bar{\tau}_m]}})+\theta({\bar{\cL}^{(0)}})$, where $\theta({\bar{X}_{[0,\bar{\tau}_m]}})$ and $\theta({\bar{\cL}^{(0)}})$ have the obvious meaning. Denote by $\bar{X}^n$ the process $\bar{X}$ conditioned on $\bar{\cN}=n$. Then $\bar{X}^n$ has the same law as the jump chain of $X_{[0,\tau_u]}$ conditioned on $\cN^{(u)}=n$. The following proposition plays a central role in the inversion. We present a combinatorial proof later.
\begin{proposition}\label{chainsjiaqiang}
The law of $\bar{X}^n$ can be described as follows. It starts from $x_0$. Conditionally on $\left(\bar{X}^n_k:0\le k\le l\right)$ with $l<T^{\bar{X}^n}$ and $\bar{X}^n_l=x$, 
\begin{itemize}
\item if $\widecheck{\Theta}_l(x)>0$, it jumps to a neighbour $y$ of $x$ with probability 
\[\frac{\widecheck{\Theta}_l(xy)}{\widecheck{\Theta}_l(x)},\]
\item if $\widecheck{\Theta}_l(x)=0$ and $x\neq x_0$, it jumps to the parent of $x$; (This can only happen when $\alpha=0$.)
\end{itemize}
And finally $\bar{X}^n$ stops at time $T^{\bar{X}^n}$. Here
\begin{align*}
    &\Theta_l(xy)=\Theta_l(n,\bar{X}^n_{[0,l]})(xy):=n(xy)-\#\{0\le k\le l-1:\bar{X}^n_k=x,\,\bar{X}^n_{k+1}=y\},\\
    &T^{\bar{X}^n}:=\inf\{k\ge 0:\Theta_k(x_0)=0\}.
\end{align*}
\end{proposition}

Now turn to step (2), i.e. recovering the jump times. 
\begin{proposition}\label{addjumptimes}
Given $\phi^{(u)}=\lambda$, $\cN^{(u)}=n$ and the jump chain of $X_{[0,\tau_u]}$, denote by $r(x)$ the number of visits to $x$ by the jump chain and  $h_i^x$ the $i$-th holding time at vertex $x$. Note that $r(x)=n(x)$ for each $x$ in the case of $\alpha=0$.
The followings hold:
\begin{longlist}
       \item $\big(\sum_{j=1}^i h^{x_0}_j:1\le i\le n(x_0)\big)$ have the same law as $n(x_0)$ i.i.d. uniform random variables on $[0,u]$ ranked in ascending order and $h^{x_0}_{n(x_0)+1}=u-\sum_{j=1}^{n(x_0)} h^{x_0}_j$;
    \item for any $x(\neq x_0)$ visited by $X_{[0,\tau_u]}$, $\left(\frac{h^x_1}{\lambda(x)},\cdots,\frac{h^x_{r(x)}}{\lambda(x)},1-\frac{\sum_{j=1}^{r(x)}h^x_j}{\lambda(x)}\right)$ is a Dirichlet distribution with parameter $(1,\cdots,1,n(x)-r(x)+\alpha).$
\end{longlist}

Here $m$-variable Dirichlet distribution with parameter $(1,\cdots,1,0)$ is interpreted as $(m-1)$-variable Dirichlet distribution with parameter $(1,\cdots,1)$. 
\end{proposition}

\begin{proof}
$\Big\{\sum_{j=1}^i h^{x_0}_j:1\le i\le n(x_0)\Big\}$ is distributed as the jump times of a Poisson process on $[0,u]$ conditioned on that there are exactly $n(x_0)$ jumps during this time interval. That deduces (\romannumeral1). 

For $x\neq x_0$, $\left\{h^x_j: 1\le j\le r(x)\right\}$ has the law of $r(x)$ i.i.d exponential variables with parameter $C_x=k_x+\sum_{y:y\sim x}C_{xy}$ conditioned on the sum of them and an independent Gamma($n(x)-r(x)+\alpha$, $C_x$)  random variable being $\lambda(x)$ (recall that a Gamma($0,\beta$) r.v is identically equal to $0$). Here we use the fact that every visit, either by $X_{[0,\tau_u]}$ or $\cL^{(0)}$, is accompanied with an exponential holding time with parameter $C_x$, and the accumulated local time of the one-point loops at $x$ provides a Gamma($\alpha$, $C_x$) time duration.
\end{proof}
Proposition \ref{chainsjiaqiang} and \ref{addjumptimes} give a representation of the inversion of the Ray-Knight identity in terms of its jump chain and holding times. Using this, we can readily calculate the jump rates.

\begin{proof}[Proof of Theorem \ref{rayknightcrossing}]
It suffices to show that the jump chain and holding times of $X^{\lambda,n}$ are given by Proposition \ref{chainsjiaqiang} and \ref{addjumptimes} respectively.
As shown in the proof of Theorem \ref{existunique}, $X^{\lambda,n}$ can be realized by a sequence of i.i.d exponential random variables. It is direct from this realization that the jump chains of $X^{\lambda,n}$ coincides with $\bar{X}^n$ in Proposition \ref{chainsjiaqiang}. To see this, one only needs to note that for any fixed $t\ge 0$ and $\omega\in \Omega$ with $T^\omega>t$ and $\omega_t=x$, the jump rate $r^{\lambda,n}_t(x,y,\omega)$ is proportional to $\widecheck{\Theta}_t(n,\omega)(xy)$ for $y\sim x$. So it remains to check the holding times. We consider $X_{[0,\tau_u]}$ given $\phi^{(u)}=\lambda$ and $\cN^{(u)}=n$. Let $h^x_i$ be the $i$-th holding time at $x$. 
For $1\le i\le n(x)$, set $l^x_i:=\sum_{j=1}^i h^{x}_j$.  By Proposition \ref{addjumptimes}, conditionally on all the holding times before, it holds that $h^{x}_{i+1}/(\lambda(x)-l^x_i)$ follows a Beta$(1,\widecheck{n}(x)-i+1)$ distribution. We can readily check that 
\[(\lambda(x)-l^x_i)\cdot \text{Beta}(1,\widecheck{n}(x)-i+1)\overset{d}{=}\big(u_x^{(i)}\big)^{-1}(\gamma),\] where $u_x^{(i)}(t)=\int_0^t \frac{\widecheck{n}(x)-i+1}{\lambda(x)-l^x_i-s}\d s$ ($0\le t< \lambda(x)-l^x_i$) and $\gamma$ is an exponential random variable with parameter $1$. This is exactly the same holding times as $X^{\lambda,n}$.
We have thus proved the theorem.
\end{proof}

The remaining part is devoted to the proof of Proposition \ref{chainsjiaqiang}.

\textbf{Case of $\alpha=0$.} Condition on $\left(\bar{X}^n_k:0\le k\le l\right)$ with $l<T^{\bar{X}^n}$ and $\bar{X}^n_l=x$. It holds that the remaining path of $\bar{X}$ is completed by uniformly choosing a path with edge crossings $\Theta_l$, since the conditional probability of any such path are the same, which equals
\[\prod_{yz\in \vec{E}}(C_{yz}/C_y)^{\Theta_l(yz)}.\]
So for the probability of the next jump, it suffices to count for any $y\sim x$, the number of all possible paths with edge crossings $\Theta_l$ and the first jump being to $y$. Here we use the same idea as the proof of \cite[Proposition 2.1]{huang2018explicit}. This number equals the number of relative orders of exiting each vertex satisfying (1) the first exit from $x$ is to $y$; (2) the last exit from any $z\neq x_0$ is to $\mathfrak{p}(z)$. In particular, it is proportional to the number of relative orders of exiting $x$ satisfying the above conditions at vertex $x$, which equals
\begin{align*}
    \left\{\begin{array}{ll}
        \frac{\left(\Theta_l(x)-1-1_{\{x\neq x_0\}}\right)!}{\prod_{z:z\sim x}\left(\Theta_l(xz)-1_{\{z=\mathfrak{p}(x)\}}-1_{\{z=y\}}\right)!},&\text{ if }\Theta_l(x)\ge 2;\\
        1_{\{y=\mathfrak{p}(x)\}},&\text{ if }\Theta_l(x)=1,
    \end{array}\right.
\end{align*}
where $(-1)!:=\infty$. It follows that the conditional probability that $\bar{X}^n_{l+1}=y$ is
\begin{align}\label{tranalpha0}
    \begin{cases}
\dfrac{\widecheck{\Theta}_l(xy)}{\widecheck{\Theta}_l(x)},&\text{ if }\Theta_l(x)\ge 2;\\
1_{\{y=\mathfrak{p}(x)\}},&\text{ if }\Theta_l(x)=1.
\end{cases}
\end{align}

\textbf{Case of $\alpha>0$.} \textbf{(1) The conditional transition probability at $x_0$.} We will first deal with the conditional transition probability given $\left(\bar{X}^n_k:0\le k\le l\right)$ with $l<T^{\bar{X}^n}$ and $\bar{X}^n_l=x_0$. We further condition on the remaining crossings of $\bar{X}^n$, i.e. $\theta^\prime:=\theta({\bar{X}^n_{[l,T^{\bar{X}^n}]}})$. Note that $\Theta_l=\theta^\prime+\theta({\bar{\cL}^{(0)}})$. In particular, it holds that $\theta^\prime(x_0y)=\Theta_l(x_0y)$ for any $y\sim x_0$. By \eqref{tranalpha0}, the conditional probability that $\bar{X}^n_{l+1}=y$ is 
\[\dfrac{\theta^\prime(x_0y)}{\theta^\prime(x_0)}=\dfrac{\Theta_l(x_0y)}{\Theta_l(x_0)},\]
which is independent of the further condition and hence gives the conclusion. 

\textbf{(2) The conditional transition probability at $x\neq x_0$.} Recall that the law of the Ray-Knight triple is independent of $k_{x_0}$. In this case, it is easier to consider the process $X$ killed at $x_0$ with rate $1$. With an abuse of notation, we still use $X$ and $\cL$ to denote this process and its associated loop soup respectively\footnote{We mention that such notations are only used in this part.}. The main idea is that the recovery of the Markovian path can be viewed as the recovery of the discrete loop soup instead. We choose to work on the extended graph $\T^K$ to make sure that the probability of a loop configuration
is proportional to $\alpha^{\#\text{loops in the configuration}}$ as $K\rightarrow\infty$. When choosing an outgoing crossing at $x$ right after a crossing to $x$, there is  a unique choice that leads to one more loop in the configuration than others. This unique choice is always to $\mathfrak{p}(x)$. So the conditional transition probability from $x$ to $y$ is proportional to $\widecheck{\Theta}(xy)$. Let us make some preparations first.

\textbf{(a) Concatenation process $\mathbf{L}$.}
First we give a representation of the path $\bar{X}^n$ as a concatenation of loops in a loop soup as follows. Let $\bar{\cL}$ be the discrete loop soup associated to $\cL$.  
We focus on the loops in $\bar{\cL}$ that visits $x_0$. Denote by $\mathfrak{K}$ the number of such loops. For each of them, we uniformly and independently root it at one of its visits at $x_0$. Then we choose uniformly at random (among all $\mathfrak{K}!$ choices) an order for the rooted loops labelled in order by $\{\bar{\gamma}_i:1\le i\le \mathfrak{K}\}$ and concatenate them:
\[\mathbf{L}:=\bar{\gamma}_1\circ \bar{\gamma}_2\circ \cdots\bar{\gamma}_\mathfrak{K}.\]
We call $\bL$ the concatenation process of $\bar{\cL}$. It can be easily deduced from the properties of discrete loop soup that the path between consecutive visits of $x_0$ in any $\bar{\gamma}_i$ has the same law as an excursion of $\bar{X}$ at $x_0$. Thus, conditionally on $\theta(\bL)(x_0)=m$, $\bL$ has the same law as $\bar{X}_{[0,\bar{\tau}_m]}$. Consequently, we have the following corollary.

\begin{corollary}\label{jumpchainln}
Given a network $n\in\mathfrak{N}$, denote by $\bL^n$ the process $\bL$  conditioned on $\theta({\bar{\cL}})=n$. Then $\bL^n$ has the same law as $\bar{X}^n$.
\end{corollary}

\textbf{(b) Pairing on the extended graph.} To further explore the law of $\bL^n$, we need to introduce the extended graphs of $\T$ and the definition of  `pairing'.
Let $K\in \z^+$. Replace each edge of $\T$ by $K$ copies. The graph thus obtained, denoted by $\T^K=(V,E^K)$, is an extended graph (of $\T$). The collection of all directed edges in $\T^K$ is denoted by $\vec{E}^K$. The graph $\T^K$ is equipped with the killing measure $\delta_{x_0}$, the Dirac measure at $x_0$, and for any $x\sim y$, the conductance on each one of the $K$ directed edges from $x$ to $y$ is $C^K_{xy}:=C_{xy}/K$.  Any element in $\N^{\vec{E}^K}$ is called a network on $\T^K$. We will use $N$ to denote a deterministic network on $\T^K$. For a network $N$ on $\T^K$, the projection of $N$ on $\T$ is a network on $\T$ defined by:
\[
N^\T(xy):=\sum_{\vec{e}\in \vec{E}^K}N(\vec{e})1_{\{\vec{e}\text{ is from $x$ to $y$}\}}.
\]
In the following, we only focus on the network $N\in \{0,1\}^{\vec{E}^K}$. For such a network $N$, denote $[N]:=\{\vec{e}\in \vec{E}^K:N(\vec{e})=1\}$. For simplicity, we omit the superscript `$\rightarrow$' for directed edges throughout this part. 

\begin{definition}
Given a sourceless network $N\in \{0,1\}^{\vec{E}^K}$, a pairing of $N$ is defined to be a bijection from $[N]$ to $[N]$, such that for any $e\in[N]$, the image of $e$ is a directed edge whose head is the tail of $e$. 
\end{definition}

Given $N\in \{0,1\}^{\vec{E}^K}$, a pairing $b$ of $N$ and a subset $[N_0]$ of $[N]$, $b\vert _{[N_0]}=\{b(e): e\in [N_0]\}$ determines a set of loops and bridges. Precisely, for each $e\in [N_0]$, following the pairing on $e$, we arrive at a new (directed) edge $b(e)$. Continuing to keep track of the pairing on the new edge $b(e)$, we arrive at another edge $b(b(e))$. This procedure stops when arriving at either the initial edge $e$ again, or an edge $\in[N]\setminus[N_0]$ because we lose the information about the pairing on $[N]\setminus[N_0]$. In the former case, a loop is obtained. In the latter case, we get a path whose first edge is $e$ and last edge $\in[N]\setminus[N_0]$. Any two such paths are either disjoint, or one is a part of the other. This naturally determines a partial order.
All the maximal elements with respect to this partial order and the loops obtained in the former case form the set of bridges and loops determined by $b\vert _{[N_0]}$. 

Now we start the proof. Let $X^K$ be the CTMC on $\T^K$ starting from $x_0$ induced by the conductances $\left(C^K_{xy}\right)$ and killing measure $\delta_{x_0}$. Consider the discrete loop soup $\bar{\cL}^K$ associated to $X^K$, the projection of which on $\T$ has the same law as $\bar{\cL}$. Let $\bL^K$ be its concatenation process. By Corollary \ref{jumpchainln}, it suffices to consider the law of (the projection of) $\bL^K$ given $\left(\theta(\bar{\cL}^K)\right)^\T=n$.



Observe that conditionally on $\left(\bL^K_k:0\le k\le l\right)$ with $\bL^K_l=x\,(\neq x_0)$, for the next step, it will definitely jump along its present loop in $\bar{\cL}^K$, which is a loop visiting both $x$ and $x_0$.  Therefore, we focus on $\bar{\cL}^{K,x}:=\{\gamma\in \bar{\cL}^K:\gamma\text{ visits }x\}$. In the following, loop configurations always refer to configurations consisting of only loops visiting $x$. Note that for any $n\in \mathfrak{R}$, given $\left(\Theta({\bar{\cL}^K})\right)^\T=n$, the probability that $\bar{\cL}^K$ uses every edge at most once tends to $1$ as $K\rightarrow\infty$. So by the standard arguments (see \cite{Werner2016,Werner2020}), it suffices to show that for any $N\in \{0,1\}^{\vec{E}^K}$ with $N^\T\le n$ and $N^\T(xy)= n(xy)$ for any $y\sim x$, when further conditioned on $\theta({\bar{\cL}}^{K,x})=N$, the transition probability of $\bL^K$ at $x$ is given by the statements in the proposition. It is easily seen that due to the independence of ${\bar{\cL}}^{K,x}$ and ${\bar{\cL}}^{K}\setminus {\bar{\cL}}^{K,x}$, the transition probability at $x$ is independent of the earlier condition $\left(\theta({\bar{\cL}}^{K})\right)^\T=n.$

A key observation is that conditionally on $\theta({\bar{\cL}}^{K,x})=N$, the probability of a loop configuration of $\bar{\cL}^{K,x}$ is proportional to $\alpha^{\#\text{loops in the configuration}}$. In fact, $\bar{\cL}^{K,x}$ is a Poisson random measure on the space of loops visiting $x$. The
intensity of a loop is the product of the conductances of the edges divided by the multiplicity of the loop. Note that for any configuration $\mathscr{L}$ with $\theta({\mathscr{L}})=N$, the multiplicities of the loops in $\mathscr{L}$ are all $1$. Hence, the probability that $\bar{\cL}^{K,x}=\mathscr{L}$ is 
\begin{align*}
    P^x_{\emptyset}\cdot\alpha^{\#\text{loops in $\mathscr{L}$}}\prod_{yz\in \vec{E}}(C_{yz}/KC_y)^{N^\T(yz)},
\end{align*}
where $P^x_{\emptyset}$ is the probability that $\bar{\cL}^{K,x}$ is empty. 

Let $\mathfrak{b}$ be the pairing of $[N]$ induced by $\bar{\cL}^{K,x}$. Namely, $\mathfrak{b}(e)$ is defined to be the edge right after $e$ in the same loop in  $\bar{\cL}^{K,x}$.
We are interested in the conditional law of $\mathfrak{b}\left(\bL^K_{l-1}\bL^K_l\right)$ given the path $\bL^K_{[0,l]}$. 

Let
\begin{align*}
    \left\{\begin{aligned}
    [N_1]&=\Big\{e\in [N]:e\text{ is crossed by } \bL^K_{[0,l-1]}\text{ and the tail of $e$ is }x\Big\};\\
    [N_2]&=\Big\{e\in[N]:\text{the tail of }e\text{ is not }x\Big\}.
    \end{aligned}\right.
\end{align*}

\begin{figure}
    \centering
   \subfigure[the loops and bridges determined by $\mathfrak{b}\vert _{[N_1]\cup[N_2]}$]{\label{figa} \includegraphics[height=3cm,width=7cm]{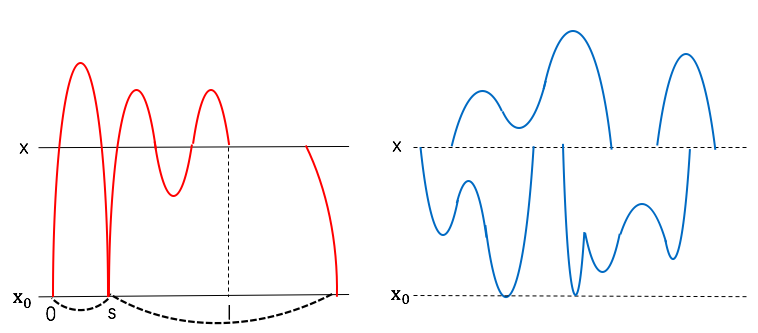}}
   \subfigure[cut off loops at $x$]{\label{figb} \includegraphics[height=3cm,width=7cm]{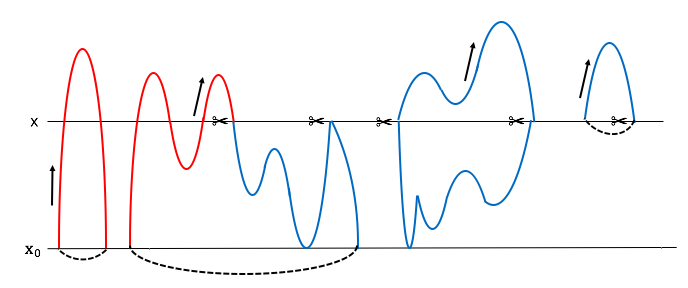}}
   \caption{(a) $\mathfrak{b}\vert _{[N_1]\cup[N_2]}$ determines some crossed loops (the first red loop) visiting $x$, one partly explored bridge (the second red bridge) from $x$ to $x$ and $(N_l^{\T}(x)-1)$ unexplored excursions (blue line on the right) from $x$ to $x$. Our aim is to find the missing piece right after $\bL^K_{[s,l]}$ in the present unfinished loop. The loop is completed either by directly gluing the second red bridge, or adding  some other blue bridges in between. (b) For any loop in $\cL^{K,x}$ that is not or partly explore by $\bL^K_{[0,l]}$, we cut the unexplored part of the loop at every visit at $x$ (including the visit at time $l$). This gives the bridges related to the loop. }
\end{figure}

\noindent Note that the path $\bL^K_{[0,l]}$ already determines $\mathfrak{b}\vert_{[N_1]}$. We further condition on $\mathfrak{b}\vert _{[N_2]}$. It holds that $\mathfrak{b}\vert_{[N_1]\cup[N_2]}$ determines a set of loops and bridges from and to $x$ (see Figure \ref{figa}). The loops are exactly the loops in $\cL^{K,x}$ that have been completely crossed by $\bL^K_{[0,l]}$, and the bridges are partitions of the remaining loops in $\cL^{K,x}$ (see Figure \ref{figb}). Denote $N_l(e)=1_{\left\{e\in [N] \text{ and is not crossed by }\bL^K_{[0,l]}\right\}}$ for $e\in \vec{E}^K$. Then there are exactly $N^\T_l(x)$ bridges, including exactly $N^\T_l(xy)$ bridges whose first edge enters $y$, for any $y\sim x$. Moreover, there is exactly one bridge partly crossed by the path $\bL^K_{[0,l]}$. This bridge is a part of the loop that $\bL^K$ is walking along at time $l$. So it visits $x_0$ by the construction of $\bL^K$, which implies that the first edge of the bridge enters $\mathfrak{p}(x)$ due to the tree structure. 

Now we focus on the conditional law of the pairing of these bridges, i.e. the law of $\mathfrak{b}^*:=\mathfrak{b}\vert _{[N^*]}$, where $[N^*]=[N]\setminus([N_1]\cup[N_2])$ is the collection of the last edges in the above $N^\T_l(x)$ bridges. Note that $\mathfrak{b}^*([N^*])$ consists of the first edges in these bridges. Denote $[N^*]=\{e_1,\cdots,e_r\}$ and $\mathfrak{b}^*([N^*])=\{f_1,\cdots,f_r\}$, where $r=N^\T_l(x)$ and we assign the subscripts such that 
\begin{itemize}
    \item $e_i$ and $f_i$ are in the same bridge ($i=1,\cdots,r$);
    \item $e_1=\bL^K_{l-1}\bL^K_l$. (So the bridge containing $e_1$ and $f_1$ is exactly the unique bridge partly crossed by $\bL^K_{[0,l]}$.)
\end{itemize} 
The totality of bijections from $\{e_1,\cdots,e_r\}$ to $\{f_1,\cdots,f_r\}$ is denoted by $\mathcal{B}$. Every $b\in\mathcal{B}$ pairs the bridges into a loop configuration. We simply call it the configuration completed by $b$. It is easy to see that this defines a one-to-one correspondence between $\mathcal{B}$ and all the possible configurations obtained by pairing these bridges.

Recall that the probability of a loop configuration of $\cL^{K,x}$ is proportional to $\allowbreak\alpha^{\#\text{loops}}$. So the conditional probability that $\mathfrak{b}^*$ equals a fixed $b\in\mathcal{B}$ is proportional to $\alpha^{\#(b)}$, where $\#(b):=\#\text{loops in the configuration completed by }b$. Set $\mathcal{B}_i:=\{b\in\mathcal{B}:b(e_1)=f_i\}$. For $1\le i,j\le r$ with $i\neq j$, a bijection $\Upsilon_{ij}$ from $B_i$ to $B_j$ can be defined as follows. For any $b\in B_i$, $\Upsilon_{ij}(b)$ is defined by exchanging the image of  $e_1$ and $b^{-1}(f_j)$. Precisely, 
\begin{align*}
\left(\Upsilon_{ij}(b)\right)(e):=\left\{
    \begin{array}{ll}
        f_j,&\text{ if }e=e_1;\\
        f_i,&\text{ if }e=b^{-1}(f_j);\\
       b(e),&\text{ otherwise}.
    \end{array}\right.
\end{align*}
We can readily check that when $x\neq x_0$, for $b\in B_i$,
\begin{align*}
    \left\{\begin{array}{ll}
        \#(b)=\#(\Upsilon_{ij}(b)),&\text{ if $i,j\neq 1$;}\\
        \#(b)=\#(\Upsilon_{ij}(b))+1,&\text{ if $i=1$ and $j\neq 1$}.
    \end{array}\right.
\end{align*}
Hence, if we denote by $p_i$ the conditional probability that $\mathfrak{b}^*\in B_i$, then 
$p_i=p_j$ for $i,j\neq 1$, and $ p_1=\alpha p_j$ for $j\neq 1$. Thus,
\begin{align*}
    \left\{\begin{aligned}
    p_1&=\alpha/(r+\alpha-1);\\
    p_i&=1/(r+\alpha-1),\text{ for }i\neq 1.
    \end{aligned}\right.
\end{align*}
It follows that the conditional probability of $\bL^K_{l+1}=y$ is 
\begin{align*}
    \sum_{i=1}^{N_l(x)}1_{\{\text{the tail of }f_i\text{ is }y\}}\cdot p_i=\dfrac{\widecheck{N}^\T_l(xy)}{\widecheck{N}^\T_l(x)}=\dfrac{\widecheck{\Theta}_l(n,\bL^K_{[0,l]})(xy)}{\widecheck{\Theta}_l(n,\bL^K_{[0,l]})(x)},
\end{align*}
where the path $\bL^K_{[0,l]}$ is understood as its projection on $\T$, and the first equality is due to the fact that $f_1$ enters $\mathfrak{p}(x)$. The conditional probability is independent of the further conditioning on $b\vert _{[N_2]}$. That completes the proof.

\subsection{The representation of the inversion as a vertex repelling jump process}\label{invertingsection}
Let $\mathcal{N}^{\lambda}$ be a sourceless $\alpha$-random network associated to $\lambda$ as in Definition \ref{rnet} and $X^\lambda$ be a process distributed, conditionally on $\mathcal{N}^{\lambda}=n$, as $X^{\lambda,n}$. By Proposition \ref{initialcrossing} and Theorem \ref{rayknightcrossing}, $X^\lambda$ has the law of $X_{[0,\tau_u]}$ conditioned on $\phi^{(u)}=\lambda$. The goal of this subsection is to show the following proposition, so as to obtain Theorem \ref{invertfinite}. Recall the distribution $\mathbb{P}^\lambda_{x_0}$ introduced in \S\ref{s:inversionrk}. 
\begin{proposition}\label{samelaw}
For any $\lambda\in\mathfrak{R}$, $X^\lambda$ has the law $\mathbb{P}^\lambda_{x_0}$. 
\end{proposition}

In other words, $X^\lambda$ is a jump process on $V$ with jump rates given by \eqref{xxrate}. Recall the definition of $\Lambda_t$ and $\Theta_t$ in \eqref{ratesymbol} and \eqref{def:thetat} respectively. The key to Proposition \ref{samelaw} is the following lemma, which reveals a renewal property of the remaining crossings of $X^\lambda$, i.e. $\Theta_t({\cN^{\lambda}},X^\lambda_{[0,t]})=\Theta_t(n,X^\lambda_{[0,t]})\big\vert _{n=\cN^{\lambda}}$. The proof is given in \S\ref{proofofcondilaw}.
\begin{lemma}\label{condilaw}
     For any $\lambda\in\mathfrak{R}$,  conditionally on $t<T^{X^\lambda}$ and $\left(X^\lambda_s:0\le s\le t\right)$, the network $\Theta_t(\cN^\lambda,X^\lambda_{[0,t]})$ is an $\alpha$-random network with sources $(x_0,X^\lambda_t)$ associated to $\allowbreak\Lambda_t$.
\end{lemma}

\noindent {\it Remark}. In the statement of the proposition and below, a network with sources $(x_0,x_0)$ has to be understood as a sourceless network. 

\begin{proof}[Proof of Proposition \ref{samelaw}]
By Lemma~\ref{condilaw}, for $\lambda\in \mathfrak{R}$, conditionally on $t<T^{X^\lambda}$ and $\left(X^\lambda_s:0\le s\le t\right)$ with $X^\lambda_t =x$,  for any $y\sim x$, 
\begin{itemize}
    \item if $x=\mathfrak{p}(y)$, $\Theta_t(\cN^{\lambda})(xy)$ follows a Bessel~($\alpha-1$, $\varphi_t(xy)$) distribution;
    \item if $y=\mathfrak{p}(x)$, $\Theta_t(\cN^{\lambda})(xy)-1$ follows a Bessel~($\alpha$, $\varphi_t(xy)$) distribution.
\end{itemize}
Note that if further conditioned on $\Theta_t(\cN^{\lambda})$, the process jumps to $y$ at time $t$ with rate \[\dfrac{\widecheck{\Theta}_t(\cN^{\lambda})(xy)}{\Lambda_t(x)}.\] So using Corollary \ref{accuratebound} and averaging over $\Theta_t(\cN^{\lambda})(xy)$, we get the probability of a jump of $X^\lambda$ from $x$ to $y$ during $[t,t+\Delta t]$ is
\begin{small}
\begin{align*}
\begin{cases}
 \Big(I_\alpha\left(\varphi_t(xy)\right)^{-1}\sum_{k\geq 0}\dfrac{k+\alpha}{\Lambda_t(x)}\dfrac{\left(\varphi_t(xy)/2\right)^{2k+\alpha}}{k!\cdot\Gamma(k+\alpha+1)}\Big)\Delta t+o(\Delta t), &\text{ if $y=\mathfrak{p}(x)$};\\
\Big(I_{\alpha-1}\left(\varphi_t(xy)\right)^{-1}\sum_{k\geq 0}\dfrac{k}{\Lambda_t(x)}\dfrac{\left(\varphi_t(xy)/2\right)^{2k+\alpha-1}}{k!\cdot\Gamma(k+\alpha)}\Big)\Delta t+o(\Delta t),&\text{ if $x=\mathfrak{p}(y)$},\\
\end{cases}
\end{align*}
\end{small}which gives the jump rates \eqref{xxrate}.
\end{proof}

\subsubsection{Proof of Lemma \ref{condilaw}}\label{proofofcondilaw}
For any $x\in V$, set
\begin{align}\label{Rxdef}
    \mathfrak{R}^x:=\{\lambda\in \mathfrak{R}:\lambda(x)>0\},\ \mathfrak{N}^x:=\{n\in \mathfrak{N}:\partial n=(x_0,x)\}.
\end{align}
In particular, $\mathfrak{R}=\mathfrak{R}^{x_0}$ and $\mathfrak{N}=\mathfrak{N}^{x_0}$. First let us generalize the notations $X^\lambda$, $X^{\lambda,n}$ and $\cN^\lambda$. Forgetting the original definition before, we construct $\Big\{X^{\lambda,n}_{[0,T^{X^{\lambda,n}}]}:\lambda\in \mathfrak{R},\, n\in \bigcup_{x\in V}\mathfrak{N}^x\Big\}$, $\left\{\cN^\lambda:\lambda\in \mathfrak{R}\right\}$ and $\big\{\p_x:x\in V\big\}$ a family of random processes, random networks and probability measures respectively on the same measurable space, such that for any $x\in V$, $\lambda\in \mathfrak{R}^x$, $n\in\mathfrak{N}^x$, under $\p_x$,
\begin{itemize}
    \item $X^{\lambda,n}$ is a process that starts at $x$, has the jump rates
    $r^{\lambda,n}_t$ (as defined in \eqref{yyrate}) and the same resurrect mechanism as the vertex-edge repelling jump process, and stops at $T^{X^{\lambda,n}}$ the time when the process exhausts the local time at $x_0$ or explodes; 
    \item $\cN^\lambda$ is an $\alpha$-random network with sources $(x_0,x)$ associated to $\lambda$;
    \item $X^{\lambda}$ is a process distributed, conditionally on $\cN^{\lambda}=n$, as $X^{\lambda,n}$.
\end{itemize}
It is easy to see that for $\lambda\in \mathfrak{R}$ and $n\in\mathfrak{N}$, under $\p_{x_0}$, $X^\lambda$, $X^{\lambda,n}$ and $\cN^\lambda$ are consistent with the original definition.

We start the proof with an observation. To emphasize the tree where $X^\lambda$ is defined, let us write $X^\lambda=X^{\lambda,\T}$. Any connected subgraph $\T_0=(V_0,E_0)$ containing $x_0$ is automatically equipped with conductances $(C_e)_{e\in E_0}$ and no killing. 
 The induced CTMC is exactly the print of $X$ on $V_0$ (recall \eqref{print}).
The following restriction principle is a simple consequence of Theorem \ref{rayknightcrossing} and the excursion theory.
\begin{proposition}[Restriction principle]\label{restrictionprinciple}
For any connected subgraph $\T_0=(V_0,E_0)$ containing $x_0$, the print of $X^{\lambda,\T}$ on $V_0$ (similarly defined as \eqref{print}) has the same law as $X^{\lambda\vert _{\T_0},\T_0}$.
\end{proposition}

Observe that by the restriction principle, it suffices to tackle the case where $\T$ is finite, which we will assume henceforth.

For $\lambda\in \mathfrak{R}^x$, 
consider $X^{\lambda}$ under $\p_x$. In the following, we simply write $\Theta_t(\cN^{\lambda})$ for $\Theta_t(\cN^{\lambda},X^{\lambda}_{[0,t]})$ and denote $J_1=J_1^{X^\lambda}$. For $0<t\le \lambda(x)$, let $\lambda^t(y):=\lambda(y)-t\cdot 1_{\{y=x\}}$ for $y\in V$. We will show that under $\p_x$,
\begin{longlist}
    \item for any $0<t\le \lambda(x)$, conditionally on $X^{\lambda}_s=x$ on $[0,t]$, $\Theta_t(\cN^{\lambda})$ is an $\alpha$-random network with sources $(x_0,x)$ associated to $\lambda^t$;
    \item for any $y\sim x$, conditionally on $J_1\le\lambda(x)$ and $X^{\lambda}_{J_1}=y$,  $\Theta_{J_1}(\cN^{\lambda})$ is an $\alpha$-random network with sources $(x_0,y)$ associated to $\lambda^{J_1}$.
\end{longlist}

Note that once (\romannumeral1) and (\romannumeral2) are proved, we have conditionally on a stay or jump at the beginning of $X^{\lambda}$,
\begin{longlist}
    \item[(a)]
    the remaining crossings $\Theta_t(\cN^{\lambda})$ is distributed as an $\alpha$-random network associated to the remaining local time field;
    \item[(b)] the process in the future is distributed as $X^{\lambda^\prime}$ under $\p_y$, where $\lambda^\prime$ is the remaining local time and $y$ equals $x$ or the vertex it jumps to accordingly. 
    In fact, it is simple to deduce from the strong renewal property of $X^{\lambda,n}$ (see Corollary \ref{selfstrong}) an analogous property for $X^\lambda$ that reads as follows: for any stopping time $S$, conditionally on $S<T^{X^{\lambda}}$ and $\allowbreak\big(X^{\lambda}_t:0\le t\le S\big)$ with $X^{\lambda}_S=y$ and $\Lambda_S(\cdot,X^{\lambda}_{[0,S]})=\lambda^\prime$, the process after $S$ i.e. $\allowbreak\Big(X^{\lambda}_{S+t}:0\le t\le T^{X^{\lambda}}-S\Big)$ has the same law as $X^{\lambda^\prime,\cN^\prime}$ under $\p_y$, where $\cN^\prime$ is a random network following the conditional distribution of $\Theta_S(\cN^\lambda)$. Then the statement follows from (a).
\end{longlist}

Under $\p_{x_0}$, iteratively using (a) and (b), we have after a chain of stays or jumps, $\Theta_t(\cN^{\lambda})$ keeps to be distributed as an $\alpha$-random network associated to the remaining local time, which leads to the conclusion.

We present the proof of (\romannumeral2), and the proof of (\romannumeral1) is similar. First consider the case of $J_1=\lambda(x)$. Notice that this event has a positive probability under $\p_x$ only when $\alpha=0$, $x\neq x_0$, $y=\mathfrak{p}(x)$ and $\cN^{\lambda}(x)=\cN^{\lambda}(xy)=1$. In this case, $\cN^{\lambda}$ is a $0$-random network with sources $(x_0,x)$ associated to $\lambda$ conditioned on $\cN^{\lambda}(x)=1$. Since $\Theta_{J_1}(\cN^{\lambda})(ij)=\cN^{\lambda}(ij)-1_{\{ij=xy\}}$ for $ij\in \vec{E}$, it is easily seen that $\Theta_{J_1}(\cN^{\lambda})$ is a $0$-random network with sources $(x_0,x)$ associated to $\lambda^{J_1}$.

Now we focus on the case where $J_1<\lambda(x)$. Recall the $\alpha$-random network defined in Definition \ref{rnet}.
A simple calculation shows that the law of an $\alpha$-random network with sources $(x_0,i)$ associated to $\lambda$ is given by:
\begin{align*}
\begin{split}
K^{\lambda,i}(n)&=1_{\{\partial n=(x_0,i)\}}\cdot \sigma^i_{\alpha}(\lambda)^{-1}\sqrt{\dfrac{\lambda(x_0)}{\lambda(i)}}\cdot\prod_{x\in V} \lambda(x)^{n(x)+\frac{\alpha-1}{2}\text{deg}(x)}\prod_{xy\in \vec{E}}\frac{\left(C^*_{xy}\right)^{\widecheck{n}(xy)}}{\Gamma(\widecheck{n}(xy)+1)},
\end{split}
\end{align*}
where $\text{deg}(x):=\#\{y\in V:y\sim x\}$ and
\[\sigma^i_\alpha:=\prod_{\{x,y\}\in \mathfrak{p}(x_0,i)}I_\alpha\big(2C^*_{xy}\sqrt{\lambda(x)\lambda(y)}\big)\prod_{\{x,y\}\in  E\setminus\mathfrak{p}(x_0,i)}I_{\alpha-1}\big(2C^*_{xy}\sqrt{\lambda(x)\lambda(y)}\big).\] 

For $r\in \mathfrak{N}^y$, set $r^\prime(ij)=r(ij)+1_{\{ij=xy\}}$ for $ij\in \N^{\vec{E}}$. Then for any Borel subset $D\subset(0,u)$, 
\begin{align}\label{comparing1}
    \begin{aligned}
    &\quad\,\p_x\left(J_1\in D,\,X^{\lambda}_{J_1}=y,\,  \Theta_{J_1}(\cN^{\lambda})=r\right)\\
    &=K^{\lambda,x}(r^\prime)\cdot\p_x\Big(J_1(X^{\lambda,r^\prime})\in D,\,X^{\lambda,r^\prime}_{J_1(X^{\lambda,r^\prime})}=y\Big)\\
    &=K^{\lambda,x}(r^\prime)\int_D\left(\dfrac{\lambda^s(x)}{\lambda(x)}\right)^{\widecheck{r^\prime}(x)}\dfrac{\widecheck{r^\prime}(xy)}{\lambda^s(x)}\d s=\int_D C(\lambda,s) K^{\lambda^s,y}(r)\d s,
    \end{aligned}
\end{align}
where the second equality is due to \eqref{J1pathprob} and in the last expression,
\[C(\lambda,s):=\dfrac{K^{\lambda,x}(r^\prime)}{K^{\lambda^s,y}(r)}\left(\dfrac{\lambda^s(x)}{\lambda(x)}\right)^{\widecheck{r^\prime}(x)}\dfrac{\widecheck{r^\prime}(xy)}{\lambda^s(x)}=\dfrac{\sigma^y_\alpha(\lambda^s)}{\sigma^x_\alpha(\lambda)}\sqrt{\dfrac{\lambda(x_0)\lambda(y)}{\lambda^s(x_0)\lambda(x)}}\left(\dfrac{\lambda(x)}{\lambda^s(x)}\right)^{(1-\alpha)1_{\{x\neq x_0\}}},\]
that is independent of $r$. Summing over $\allowbreak r\in \mathfrak{N}^x$ in \eqref{comparing1}, we get
\begin{align}\label{J1monotone}
    \p_x\left(J_1\in D,\,X^{\lambda}_{J_1}=y\right)=\int_D C(\lambda,s)\d s.
\end{align}
By the monotone class methods, we can replace `$1_D$' in \eqref{J1monotone} by any non-negative measurable function on $\r^+$ vanishing on $[u,\infty)$. In particular, 
\begin{align}\label{comparing2}
\begin{aligned}
    &\e_x\left(K^{\lambda^{J_1},y}(r);J_1\in D,\,X_{J_1}^\lambda=y\right)=\int_D C(\lambda,s) K^{\lambda^s,y}(r)\d s.
\end{aligned}
\end{align}
Comparing \eqref{comparing1} and \eqref{comparing2}, we have
\begin{align*}
   \p_x\left(\Theta_{J_1}(\cN^{\lambda})=r,\,J_1\in D,\,X^{\lambda}_{J_1}=y\right)=\e_x\left(K^{\lambda^{J_1},y}(r);\,J_1\in D,\,X^{\lambda}_{J_1}=y\right). 
\end{align*}
We have thus reached (\romannumeral2).

\section{Inverting percolation Ray-Knight identity}\label{invertpercolation}

In this section, we only consider the case where the conductances are symmetric and $\alpha\in (0,1)$. In \S\ref{percolation}, we will introduce the metric-graph Brownian motion. It turns out that the metric-graph Brownian motion together with the associated percolation process gives a realization of $\cX$ defined in \S\ref{percoray}, which leads to the percolation Ray-Knight identity (Theorem \ref{percolationrayknight}). \S\ref{proofinvertrk} is devoted to the proof of the inversion of percolation Ray-Knight identity (Theorem \ref{invertfinitecluster}).

\subsection{Percolation Ray-Knight identity}\label{percolation}
Replacing every edge $e=\{x,y\}$ of $\T$ by an interval $I_e$ with length $1/2C_{xy}$, it defines the metric graph associated to $\T$, denoted by $\tG$. $V$ is considered to be a subset of $\tG$. One can naturally construct a metric-graph Brownian motion $B$ on  $\tG$, i.e. a diffusion that behaves like a Brownian motion inside each edge, performs Brownian excursion inside each adjacent edges when hitting a vertex in $V$ and is killed at each vertex $x\in V$ with rate $k_x$. Let $\wL$ be the loop soup associated to $B$. Then $X$ and $B$, resp. $\cL$ and $\wL$, can be naturally coupled through restriction (i.e. $X$, resp. $\cL$, is the print of $B$, resp. $\wL$, on $V$), which we will assume from now on. See \cite{lupu2016loop} for more details of metric graphs, metric-graph Brownian motions and the above couplings.
Notations such as $L_\cdot(\wL)$ and $\wL^{(u)}$ $(u\ge 0)$ are similarly defined for $\wL$ as for $\cL$.  Assume that $B$ starts from $x_0$. We also have the Ray-Knight identity for $B$, that is to replace $\cL, X, \tau_u$ in Theorem \ref{rayknighttransient} by $\wL, B, \tau^B_u$ respectively.

Let $B$ and $\wL^{(0)}$ be independent, and we always consider under the condition $\tau_u<T^X$ (or equivalently $\tau^B_u<T^B$). For $x\in \tG$ and $t\le \tau_u$, we set 
\[
\widetilde{\phi}_t(x):=L_x(\wL^{(0)})+L^{B}(Q^{-1}(t),x),\]
where $Q(t)=\sum_{x\in V}L^{B}(t,x)$ and $Q^{-1}$ is the right-continuous inverse.
By the coupling of $B$ and $X$, it holds that
$\widetilde{\phi}_t\vert _V=\phi_t$ (defined in \eqref{lambdatdef}).
Next, for $t\geq 0$, $\widetilde{\O}_t$ will denote a percolation on $E$ defined as: for $e\in E$,
\[
\widetilde{\O}_t(e)=1_{\left\{\widetilde{\phi}_t\text{ has no zero on }I_{e}\right\}}.\]

Recall the notations defined in \S\ref{percoray}. We will show the following proposition, which immediately implies the percolation Ray-Knight identity. 
\begin{proposition}\phantomsection\label{cxcx}
\begin{longlist}
     \item $(\widetilde{\phi}_{\tau_u},\widetilde{\O}_{\tau_u})$ has the same law as $(L_\cdot(\cL^{(u)}),\O^{(u)})$;
    \item $\widetilde{\cX}:=(X_t,\widetilde{\O}_t)_{0\le t\le \tau_u}$ has the same law as $\cX=(X_t,\O_t)_{0\le t\le \tau_u}$.
\end{longlist}
\end{proposition}

The remaining part is devoted to the proof of Proposition \ref{cxcx}.
First we present two lemmas in terms of the loop soups $\cL$ and $\wL$, which will later be translated into the analogous versions associated to $\widetilde{\cX}$ using the Ray-Knight identity.

Define
$O(\cL),O(\wL)$ on $E$ as follows: for $e\in E$,
\[O(\wL)(e):=1_{\{L_\cdot(\wL)\text{ has no zero on } I_{e}\}},~O(\cL)(e):=1_{\{e\text{ is crossed by }\cL\}}.\]
Remember that $\cL$ and $\wL$ are naturally coupled. So it holds that $O(\wL)\ge O(\cL)$. Moreover, they have the following relations. 
\begin{lemma}\label{olo}
Conditionally on $\cL$, $\left(O(\wL)(e):e\in E,\, O(\cL)(e)=0\right)$ is a family of  independent random variables, and
\begin{align}\label{eqolo}
\begin{split}
&\p\left(O(\wL)(\{x,y\})=0\,\big\vert\, L_\cdot(\cL)=\ell,O(\cL)(\{x,y\})=0\right)=\bK_{1-\alpha}(2C_{xy}\sqrt{\ell_x\ell_y}),
\end{split}
\end{align}
where $\bK_\nu(z):=\frac{2\left(z/2\right)^{\nu}}{\Gamma(\nu)}K_{\nu}\left(z\right)$.
\end{lemma}

 \begin{lemma}\label{ool}
Conditionally on $L_\cdot(\cL)$, $\left(O(\wL)(e):e\in E\right)$ is a family of  independent random variables, and
\begin{align}\label{ze}
\begin{split}
&\p\left(O(\wL)(\{x,y\})=1\,\big\vert\, L_\cdot(\cL)=\ell\right)=\dfrac{I_{1-\alpha}\left(2C_{xy}\sqrt{\ell_x\ell_y}\right)}{I_{\alpha-1}\left(2C_{xy}\sqrt{\ell_x\ell_y}\right)}.
\end{split}
\end{align}
\end{lemma}
\begin{proof}[Proof of Lemma \ref{olo}]
The independence follows from an argument using the Ray-Knight identity of $B$ and the excursion theory similar to that in the proof of Proposition \ref{initialcrossing}. With the same idea as \cite[\S3]{lupu2016loop}, conditionally on $O(\cL)(\{x,y\})=0$, the trace of $\wL$ in $\{x,y\}$ consists of the loops entirely contained in $I_{\{x,y\}}$, excursions from and to $x$ (resp. $y$) inside $I_{\{x,y\}}$ of loops in $\wL$ visiting $x$ (resp. $y$). By considering the contribution to the occupation time field by each part, we have that the left-hand side of \eqref{eqolo} is the same as the probability that the sum of three independent processes
\begin{align}\label{zee}
\left(a_t^{(h)}+b^{(h,\ell_x)}_t+b_{h-t}^{(h,\ell_y)}\right)_{0\le t\le h}
\end{align}
has a zero on $(0, h)$. Here $h=1/2C_{xy}$, $(a_t^{(h)})_{0\le t\le h}$ is a $\text{BESQ}^{2\alpha,h}_{0\rightarrow 0}$ (i.e. a $2\alpha$-dimensional BESQ bridge from $0$ to $0$ over $[0,h]$) and $(b_t^{(h,l)})_{0\le t\le h}$ is a $\text{BESQ}^{0,h}_{l\rightarrow 0}$.

For \eqref{zee} to have a zero, the process $b^{(h,\ell_x)}_t$ has to hit $0$ before the last zero of $(a_t^{(h)}+b_{h-t}^{(h,\ell_y)})_{0\le t\le h}$. The density of the first zero of  $b^{(h,\ell_x)}_t$ is

\begin{align}\label{be}
1_{\{0<t<h\}}\dfrac{\ell_x}{2t^2}\exp\left(-\frac{(h-t)\ell_x}{2ht}\right)\d t.
\end{align}
To get this, one can start with the well-known fact that the first zero of a $2\delta$-dimensional BESQ process starting from $x$ is distributed as $x^2/2\text{Gamma}(1-\delta,1)$ (Cf. for example \cite[Proposition~2.9]{lawler2018notes}). Then use the fact that for a $2\delta$-dimensional BESQ process $\rho_t$ starting from $x$, the process $h(1-u/h)^2\rho_{u/(h-u)}$ $(0\le u\le h)$ is a $\text{BESQ}^{2\delta,h}_{x\rightarrow 0}$.

Since $(a_t^{(h)}+b_{h-t}^{(h,\ell_y)})_{0\le t\le h}$ has the same law as a $\text{BESQ}^{2\alpha,h}_{0\rightarrow\ell_y}$, its last zero has the same law as $h$ minus the first zero of a $\text{BESQ}^{2\alpha,h}_{\ell_y \rightarrow 0}$, which has the density
\begin{align}\label{lz}
    1_{\{0<t<h\}}\dfrac{2^{\alpha-1}}{\Gamma(1-\alpha)}h^\alpha \ell_y^{1-\alpha}t^{-\alpha}(h-t)^{\alpha-2}\exp{\left(-\frac{\ell_yt}{2h(h-t)}\right)}\d t.
\end{align}

 
Gathering \eqref{be}, \eqref{lz} and taking $h=1/2C_{xy}$, we get the probability of \eqref{zee} having a zero is 
\begin{align*}
&\int_0^h\int_0^t\dfrac{\ell_x}{2s^2}\exp\left(-\frac{(h-s)\ell_x}{2hs}\right)\dfrac{2^{\alpha-1}}{\Gamma(1-\alpha)}h^\alpha \ell_y^{1-\alpha}t^{-\alpha}(h-t)^{\alpha-2}\exp{\left(-\frac{\ell_yt}{2h(h-t)}\right)}\d s\d t\\
&=\dfrac{2^{\alpha-1}}{\Gamma(1-\alpha)}\ell_y^{1-\alpha}h^\alpha\exp\left(\dfrac{\ell_y+\ell_x}{2h}\right)\int_0^h\exp\left(-\dfrac{\ell_x}{2t}-\dfrac{\ell_y}{2(h-t)}\right)t^{-\alpha}(h-t)^{\alpha-2}\d t\\
&=\dfrac{1}{\Gamma(1-\alpha)}\int_0^\infty\exp\left(-\frac{\ell_x\ell_y}{(2h)^2s}-s\right)\dfrac{\d s}{s^\alpha}\\
&=2\Gamma(1-\alpha)^{-1}\left(C_{xy}\sqrt{\ell_x\ell_y}\right)^{1-\alpha}K_{1-\alpha}\left(2C_{xy}\sqrt{\ell_x\ell_y}\right),
\end{align*}
where in the second equality, we use the change of variable $s=\dfrac{\ell_yt}{2h(h-t)}$ and the last equality follows from \cite[(136)]{Yor}.
\end{proof}

\begin{proof}[Proof of Lemma \ref{ool}]
By Proposition \ref{initialcrossing}, 
\begin{align}\label{b0}
\p\left(O(\cL)(\{x,y\})=0 \vert L_\cdot(\cL)=\ell\right)=\dfrac{\left(C_{xy}\sqrt{\ell_x\ell_y}\right)^{\alpha-1}}{\Gamma(\alpha)I_{\alpha-1}(2C_{xy}\sqrt{\ell_x\ell_y})}.
\end{align}
The probability we are interested in equals $1$ minus the multiplication of \eqref{eqolo} and \eqref{b0}.
\end{proof}
\begin{proof}[Proof of Proposition \ref{cxcx}]
By the Ray-Knight identity of $B$, we have $\widetilde{\phi}_{\tau_u}$ has the law of $L_\cdot(\wL^{(u)})$. So it follows from Lemma \ref{ool} that $(\widetilde{\phi}_{\tau_u},\widetilde{\O}_{\tau_u})$ has the same law as $(L_\cdot(\cL^{(u)}),\O^{(u)})$. That concludes (1).

For (2), if $X$ jumps through the edge $\{x,y\}$, then $B$ crosses the interval $I_{\{x,y\}}$, which makes $\widetilde{\phi}_t$ positive on this interval. So $\widetilde{\O}_t(\{x,y\})$ turns to $1$.

It remains to calculate the rate of opening an edge without $X$ jumping. First it is easy to deduce that $(\widetilde{\cX},\phi)$ is a Markov process, which
allows us to condition only on $(\widetilde{\cX}_t,\phi_t)$ when calculating the rate. Note that the case of opening $\{x,y\}$ without $X$ jumping can happen only when $X$ stays at $x$ and $x=\mathfrak{p}(y)$. We use $\{X_{[t,t+\Delta t]}=x\}$ to stand for the event that $X$ stays at $x$ during $[t,t+\Delta t]$. Denote by $A_1$ the event that at some time in $[t,t+\Delta t]$, $\{x,y\}$ is opened without $X$ jumping. And let $A_2:=A_1\cap \{X_{[t,t+\Delta t]}=x\}$. For $(\widetilde{\cX}_t,\phi_t)$ with $X_t=x$ and $\widetilde{\O}_t(\{x,y\})=0$, if we write $\q=\p(\cdot\,\vert \,X_t=x,\phi_t)$, then
\begin{align*}
    &\quad\,\p(A_2\,\vert \,\widetilde{\cX}_t,\phi_t)=\q( \widetilde{\O}_{t+\Delta t}(\{x,y\})=1, X_{[t,t+\Delta t]}=x \,\vert \,\widetilde{\O}_t(\{x,y\})=0)\\
    &=\q(\widetilde{\O}_t(\{x,y\})=0)^{-1}\cdot\Big[ \q(\widetilde{\O}_t(\{x,y\})=0, X_{[t,t+\Delta t]}= x)\\
    &\quad-\q(\widetilde{\O}_{t+\Delta t}(\{x,y\})=0, X_{[t,t+\Delta t]}=x)\Big]\\
    &=\q(X_{[t,t+\Delta t]}=x)\cdot\Bigg[1-\dfrac{\q(\widetilde{\O}_{t+\Delta t}(\{x,y\})=0 \,\vert \,X_{[t,t+\Delta t]}= x)}{\q(\widetilde{\O}_t(\{x,y\})=0)}\Bigg].
\end{align*}
By considering the print of $B$ on the branch at $x$ containing $y$ and using the Ray-Knight identity of $B$, we can readily check that the fraction in the above square brackets equals
\begin{align*}
    \dfrac{\p\big(O(\wL)(\{x,y\})=0\,\vert \,O(\cL)(\{x,y\})=0,L_x(\cL)=\phi_t(x)+\Delta t,L_y(\cL)=\phi_t(y)\big)}{\p\big(O(\wL)(\{x,y\})=0\,\vert \,O(\cL)(\{x,y\})=0,L_x(\cL)=\phi_t(x),L_y(\cL)=\phi_t(y)\big)},
\end{align*}
which further equals $f(\phi_t(x)+\Delta t)/f(\phi_t(x))$ with $f(s):=\mathbb{K}_{1-\alpha}\left(2C_{xy}\sqrt{
s\phi_t(y)}\right)$ by Lemma \ref{olo}. Therefore,
\begin{align*}
    \p(A_2\,\vert \,\widetilde{\cX}_t,\phi_t)=-\dfrac{f^\prime(\phi_t(x))}{f(\phi_t(x))}\Delta t+o(\Delta t).
\end{align*}
Observe that $A_1\setminus A_2\subset A_1\cap \{X \text{ has a jump during }[t,t+\Delta t]\}$. We can readily deduce that the conditional probability of the latter event is $o(\Delta t)$. So
\begin{align*}
    \p(A_1\,\vert \,\widetilde{\cX}_t,\phi_t)=\p(A_2\,\vert \,\widetilde{\cX}_t,\phi_t)+o(\Delta t)=-\dfrac{f^\prime(\phi_t(x))}{f(\phi_t(x))}\Delta t+o(\Delta t),
\end{align*}
which leads to the rate in \eqref{o1} by using $[z^\nu K_\nu(z)]^\prime=-z^\nu K_{\nu-1}(z)$ and $K_{-\nu}=K_\nu$.
\end{proof}



\subsection{Percolation-vertex repelling jump process}\label{proofinvertrk}
Let  $\overleftarrow{\cX}=(\overleftarrow{X}_t,\overleftarrow{\O}_t)_{0\le t\le \tau_u}$ be the time reversal of $\cX$, i.e.  $\overleftarrow{\cX}=\left(\cX_{\tau_u-t}\right)_{0\le t\le \tau_u}$. In this part, we will verify that for any $\lambda\in \mathfrak{R}$ and configuration $O$ on $E$, the jump rate of $\overleftarrow{\cX}$ conditionally on $\tau_u<T^X$ and $\left(\phi^{(u)}, \O^{(u)}\right)=(\lambda,O)$ is given by \eqref{cxrate}, which leads to Theorem \ref{invertfinitecluster}.

Recall the notations in \eqref{ratesymbol}. In the following, we will use $\Lambda_t$ and $\varphi_t$ to represent $\Lambda_t(\lambda,\overleftarrow{X}_{[0,t]})$ and $\varphi_t(\lambda,\overleftarrow{X}_{[0,t]})$ respectively. Note that given $\phi^{(u)}=\lambda$, the aggregated local time $\phi_{\tau_u-t}$ equals $\Lambda_t$. 
Let us condition on $\tau_u<T^X$, $\left(\phi^{(u)}, \O^{(u)}\right)=(\lambda,O)$ and the path $\overleftarrow{\cX}_{[0,t]}$ with $\overleftarrow{X}_t=x$, and calculate the jump rate of the edge $\{x,y\}$, i.e. the rate of the jump from $x$ to $y$ by $\overleftarrow{X}_t$ without modifying $\overleftarrow{\O}_t$, or the closure of $\{x,y\}$ in $\overleftarrow{\O}_t$ without $\overleftarrow{X}$ jumping, or the jump and closure simultaneously. Due to the Markov property of $\cX$, it suffices to condition only on $(\overleftarrow{\cX}_t,\Lambda_t)$. For simplicity, denote by $\q=\p(\cdot \,\vert \,\overleftarrow{X}_t=x,\Lambda_t)=\p(\cdot\,\vert \, X_{\tau_u-t}=x,\phi_{\tau_u-t})$ henceforth.
\subsubsection{Two-point case}\label{twopoint}
We start with the simplest case: $\T$ contains only two vertices.
The jump rate is analyzed in the following two cases.


(1) if $y=\mathfrak{p}(x)$, it holds that $\overleftarrow{\O}_t(\{x,y\})=1$. This allows us to further ignore the condition on $\overleftarrow{\O}_t(\{x,y\})$. Note that the law of $\overleftarrow{X}$  given $\phi^{(u)}=\lambda$ and $\tau_u<T^X$ is also $\p_{x_0}^\lambda$ (defined in \S\ref{qwer666}). So the rate of the jump from $x$ to $y$ by $\overleftarrow{X}$ at time $t$ is $C_{xy}\sqrt{\dfrac{\Lambda_t(y)}{\Lambda_t(x)}}\cdot\dfrac{I_{\alpha-1}(\varphi_t(xy))}{I_{\alpha}(\varphi_t(xy))}$ as shown in \eqref{xxrate}.
    
We further consider the probability that the jump is accompanied by the closure of $\{x,y\}$ in $\O$. Observe that this happens if and only if $\O_{(\tau_u-t)-}(\{x,y\})=0$. Conditionally on $\Lambda_t$ and $\overleftarrow{X}$ jumping from $x$ to $y$ at time $t$, $\O_{(\tau_u-t)-}(\{x,y\})$ has the same law as $O(\wL)(\{x,y\})$ given $L_x(\wL)=\Lambda_t(x)$ and $L_y(\wL)=\Lambda_t(y)$. By Lemma \ref{ool}, the conditional probability that
$\O_{(\tau_u-t)-}(\{x,y\})=0$ is
\begin{equation}\label{opco}
    1-\dfrac{I_{1-\alpha}(\varphi_t(xy))}{I_{\alpha-1}(\varphi_t(xy))}.
\end{equation}
Therefore,
$\overleftarrow{X}$ jumps from $x$ to $y$ and $\overleftarrow{\O}(\{x,y\})$ turns to $0$ at time $t$ with rate
 \begin{align*}
 \begin{aligned}
     &\quad C_{xy}\sqrt{\dfrac{\Lambda_t(y)}{\Lambda_t(x)}}\cdot\dfrac{I_{\alpha-1}(\varphi_t(xy))}{I_{\alpha}(\varphi_t(xy))}\cdot\Big(1-\dfrac{I_{1-\alpha}(\varphi_t(xy))}{I_{\alpha-1}(\varphi_t(xy))}\Big)\\
     &=C_{xy}\sqrt{\dfrac{\Lambda_t(y)}{\Lambda_t(x)}}\cdot\dfrac{I_{\alpha-1}(\varphi_t(xy))-I_{1-\alpha}(\varphi_t(xy))}{I_{\alpha}(\varphi_t(xy))};
 \end{aligned}
 \end{align*}
  while $\overleftarrow{X}$ jumps from $x$ to $y$ at time $t$ without modifying $\overleftarrow{\O}_t$ with rate 
        \[
        C_{xy}\sqrt{\dfrac{\Lambda_t(y)}{\Lambda_t(x)}}\cdot\dfrac{I_{1-\alpha}(\varphi_t(xy))}{I_{\alpha}(\varphi_t(xy))}.
        \]
   
   (2) if $x=\mathfrak{p}(y)$, there are two possible jumps: 
   $\overleftarrow{X}$ jumps from $x$ to $y$ without modifying $\overleftarrow{\O}_t$, or $\{x,y\}$ is closed without $\overleftarrow{X}$ jumping. Denote by $A_1$ (resp. $A_2$) the event that there is a jump of the first (resp. second) kind occurs during $[t,t+\Delta t]$. Note that $A_1$ can happen only when $\overleftarrow{\O}_t(\{x,y\})=1$. We have
    \begin{align}\label{stayprob}
    \begin{split}
        &\quad\,\p(A_1\,\vert \,\overleftarrow{X}_{t}=x,\overleftarrow{\O}_{t}(\{x,y\}=1,\Lambda_t)
        ={\q(A_1)}/{\q(\overleftarrow{\O}_{t}(\{x,y\})=1)}\\
          &= C_{xy}\sqrt{\dfrac{\Lambda_t(y)}{\Lambda_t(x)}}\cdot\dfrac{I_{\alpha}\left(\varphi_t(xy)\right)}{I_{\alpha-1}\left(\varphi_t(xy)\right)}\Delta t\cdot \Bigg(\dfrac{I_{1-\alpha}(\varphi_t(xy))}{I_{\alpha-1}(\varphi_t(xy))}\Bigg)^{-1}+o(\Delta t)\\
           &= C_{xy}\sqrt{\dfrac{\Lambda_t(y)}{\Lambda_t(x)}}\cdot\dfrac{I_{\alpha}\left(\varphi_t(xy)\right)}{I_{1-\alpha}\left(\varphi_t(xy)\right)}\Delta t+o(\Delta t).
          \end{split}
    \end{align}

Now we turn to $A_2$. Observe that $A_2$ is also the event that at some time in $[\tau_u-t-\Delta t,\tau_u-t]$, $\{x,y\}$ is opened in $\O$ without $X$ jumping. Let $A_3=A_2\cap \{X_{[\tau_u-t-\Delta t,\tau_u-t]}=x\}$. Then
 \begin{align*}\label{poiu}
    \begin{split}
    &\quad\,\p(A_3\,\vert \,\overleftarrow{X}_{t}=x,\overleftarrow{\O}_{t}(\{x,y\})=1,\Lambda_t)\\
     &=\q(\O_{\tau_u-t-\Delta t}(\{x,y\})=0,X_{[\tau_u-t-\Delta t,\tau_u-t]}=x\,\vert \,\O_{\tau_u-t}(\{x,y\})=1)\\
     &=\q(X_{[\tau_u-t-\Delta t,\tau_u-t]}=x\,\vert \,\O_{\tau_u-t}(\{x,y\})=1)\\
     &-\q(\O_{\tau_u-t-\Delta t}(\{x,y\})=1,X_{[\tau_u-t-\Delta t,\tau_u-t]}=x\,\vert \,\O_{\tau_u-t}(\{x,y\})=1)=:q_1-q_2.
          \end{split}
    \end{align*}
    We observe that  \[q_1=1-C_{xy}\sqrt{\dfrac{\Lambda_t(y)}{\Lambda_t(x)}}\cdot\dfrac{I_{\alpha}\left(\varphi_t(xy)\right)}{I_{1-\alpha}\left(\varphi_t(xy)\right)}\Delta t+o(\Delta t)\]
    by \eqref{stayprob}; while if we denote $h(s)=\frac{I_{1-\alpha}(2C_{xy}\sqrt{s\Lambda_t(y)})}{I_{\alpha-1}(2C_{xy}\sqrt{s\Lambda_t(y)})}$, then by Lemma \ref{ool},
\begin{align*}
    q_2&=q_1 \p(\O_{\tau_u-t-\Delta t}=1\,\vert \,X_{\tau_u-t-\Delta t}=x,\Lambda_{\tau_u-t-\Delta t}=\ell-\Delta t\cdot\delta_x)\vert _{\ell=\Lambda_t}\\
    &=q_1 h(\Lambda_t-\Delta t)=q_1 [h(\Lambda_t)-h^\prime (\Lambda_t)\Delta t]\\
    &=1-C_{xy}\sqrt{\dfrac{\Lambda_t(y)}{\Lambda_t(x)}}\cdot\dfrac{I_{-\alpha}\left(\varphi_t(xy)\right)}{I_{1-\alpha}\left(\varphi_t(xy)\right)}\Delta t+o(\Delta t).
\end{align*}
So $\p(A_3\,\vert \,\overleftarrow{X}_{t}=x,\overleftarrow{\O}_{t}(\{x,y\})=1,\Lambda_t)$ equals
\begin{align}\label{raterate6666}
    C_{xy}\sqrt{\dfrac{\Lambda_t(y)}{\Lambda_t(x)}}\cdot\dfrac{I_{-\alpha}(\varphi_t(xy))-I_{\alpha}(\varphi_t(xy))}{I_{1-\alpha}(\varphi_t(xy))}\Delta t+o(\Delta t).
\end{align}
Since $A_2\setminus A_3\subset \big\{\overleftarrow{X}\text{ has more than $2$ jumps during }[t,t+\Delta t]\big\}$, by Lemma \ref{twojumporder} we have the probability of $A_2\setminus A_3$ is $o(\Delta t)$ under $\q(\cdot\,\vert \,\O_{\tau_u-t}(\{x,y\})=1)$. So $\p(A_2\,\vert \,\overleftarrow{X}_{t}=x,\overleftarrow{\O}_{t}(\{x,y\})=1,\Lambda_t)$ also equals \eqref{raterate6666}. We have thus shown all the rates in \eqref{cxrate} in the two-point case. The rates already determines the process (Cf. Appendix \ref{selfinteracting}). So we are done.
    

\subsubsection{General case}
For general $\T$, the strategy is to introduce a coupling, which connects the rate in the general case with that in the two-point case. In the following, to emphasize the tree where $\cX$ is defined, we also write  $\overleftarrow{\cX}=\overleftarrow{\cX}^\T=(\overleftarrow{X}^\T,\overleftarrow{\O}^\T)$. Recall that we condition on $(\overleftarrow{\cX}_t,\Lambda_t)$ with $\overleftarrow{X}_t=x$.

If $y=\mathfrak{p}(x)$, divide $\T$ into $3$ parts: $\T_i=(V_i,E_i)~(i=1,2,3)$ is the connected subgraph of $\T$ with $V_2=\{x,y\}$, $V_3=\{z\in V:z=x\text{ or }x<z\}$, $V_1=V\setminus V_3$. Each $\T_i$ is equipped with the conductances and killing measure such that the induced CTMC on $\T_i$ has the law of the print of $X$ on $V_i$\footnote{For example, for $\T_2$, the conductance on $\{x,y\}$ is $C_{xy}$ and the killing measure on $x$ (resp. $y$) is the effective conductance (on $\T$) from $x$ (resp. $y$) to the infinity point after cutting the edge $\{x,y\}$.}.  View $x_0$, $y$ and $x$ as the roots of $\T_1$, $\T_2$ and $\T_3$ respectively. We have the following coupling. Note that $(\overleftarrow{\cX},\Lambda)=(\overleftarrow{\cX}^\T,\Lambda^\T)$ is a Markov process. So we can naturally define $(\overleftarrow{\cX},\Lambda)$ starting from a given pair $(x,O,\lambda)$, which represents the starting point, initial configuration and time field respectively. 
Now we consider $(\overleftarrow{\cX}^{\T_i},\Lambda^{\T_i})$ starting from $(x_i,\O_t\vert_{E_i},\Lambda_t\vert_{V_i})$ respectively, where $x_i=\left\{\begin{aligned}
        y,&\text{ if }i=1;\\
        x,&\text{ if }i=2,3;
    \end{aligned}\right.$.
If we glue the excursions of $\overleftarrow{X}^{\T_i}$ at $x$ and $y$ according to the local time, preserve the evolution of $\overleftarrow{O}^{\T_i}$ along with $\overleftarrow{X}^{\T_i}$, and recover a process from the glued excursion process starting from $x$ up to the time when the local time at $x_0$ is exhausted (some of the excursions are not used in the recovery procedure), then the process obtained has the law of $\big(\overleftarrow{\cX}_{t+s}:0\le s\le T^{\cX}-t\big)$ given $(\overleftarrow{\cX}_t,\Lambda_t)$.


In the following, assume $(\overleftarrow{\cX},\Lambda)$ and $(\overleftarrow{\cX}^{\T_2},\Lambda^{\T_2})$ are coupled in the above way. Let $A_1$ and $A_2$ (resp. $A^\prime_1$ and $A^\prime_2$) be the events that $\overleftarrow{X}$ (resp. $\overleftarrow{X}^{\T_2}$) jumps from $x$ to $y$ with and without the closure of $\{x,y\}$ in $\overleftarrow{\O}$ (resp. $\overleftarrow{\O}^{\T_2}$) during $[t,t+\Delta t]$ (resp. $[0,\Delta t]$) respectively. It follows direct from the coupling that for $(\overleftarrow{\cX}_t,\Lambda_t)$ with $\overleftarrow{X}_t=x$ and $\overleftarrow{\O}_t(\{x,y\})=1$,
\begin{align}\label{inequalequal}
        \p(A_j\vert \overleftarrow{\cX}_t,\Lambda_t)\le \p(A^\prime_j\vert \overleftarrow{\cX}_t,\Lambda_t)=\p\big(A^\prime_j\big\vert \overleftarrow{\cX}^{\T_2}_{Q_2(t)},\Lambda^{\T_2}_{Q_2(t)}\big)~(j=1,2),
\end{align}
where $Q_2(t)=\int_0^t 1_{\{\overleftarrow{X}_s\in V_2\}}\d s$.

Recall that $\q=\p(\cdot \,\vert \,\overleftarrow{X}_t=x,\Lambda_t)$. It is plain that
\begin{align}\label{inequaltoequal1}
\begin{aligned}
    &\quad\q(\overleftarrow{\O}_t(\{x,y\})=1)\cdot\big[\p(A_1\vert \overleftarrow{\cX}_t,\Lambda_t)+\p(A_2\vert \overleftarrow{\cX}_t,\Lambda_t)\big]\\
    &=\q\big(\text{there is a jump of $\overleftarrow{X}$ from $x$ to $y$ during }[t,t+\Delta t]\big).
\end{aligned}
\end{align}
On the other hand, the jump rates of $\overleftarrow{\cX}^{\T_2}$ have been calculated in \S\ref{twopoint}. Together with the rates of $\overleftarrow{X}$ shown in Theorem \ref{invertfinite}, we can readily verifies that 
\begin{align}\label{inequaltoequal2}
\begin{aligned}
    &\quad\,\q(\overleftarrow{\O}_t(\{x,y\})=1)\cdot\Big[\p\big(A^\prime_1\big\vert \overleftarrow{\cX}^{\T_2}_{Q_2(t)},\Lambda^{\T_2}_{Q_2(t)}\big)+\p\big(A^\prime_2\big\vert \overleftarrow{\cX}^{\T_2}_{Q_2(t)},\Lambda^{\T_2}_{Q_2(t)}\big)\Big]\\
    &=\q\big(\text{there is a jump of $\overleftarrow{X}$ from $x$ to $y$ during }[t,t+\Delta t]\big)+o(\Delta t).
\end{aligned}
\end{align}
Comparing \eqref{inequaltoequal1} and \eqref{inequaltoequal2}, we have the inequality in \eqref{inequalequal} is in fact equality up to a difference of $o(\Delta t)$. This gives the jump rates.

The jumps rates in the case of $x=\mathfrak{p}(y)$ can also be deduced from the jump rates in the two-point case using a similar argument. We safely leave the details to the readers.

\section{Mesh limits of repelling jump processes}\label{convergencesection}
In this section, we first introduce the  self-repelling diffusion $W^\lambda$ which inverts the Ray-Knight identity of reflected Brownian motion, and then show that $W^\lambda$ is the mesh limit of vertex repelling jump processes as stated in Theorem \ref{convergence}. 

\subsection{Self-repelling diffusions}

Let $\cZ$ be the process whose local time flow is $\text{Jacobi} (2\alpha,0)$ flow, which is the so-called burglar. See \cite[\S 5]{Aidekon} and \cite{warren98} for details.
 
 For any non-negative continuous function $f$ on $\r^+$ with $f(0)>0$, denote by $\mathfrak{d}_f$ the hitting time of $0$ by $f$, i.e. $\mathfrak{d}_f:=\inf\{x>0:f(x)=0\}\in(0,\infty]$. We call $f$ admissible if
\[\int_0^{\mathfrak{d}_f} f(x)^{-1}\d x=\infty.\]
Let $\widetilde{\mathfrak{R}}$ be the set of  non-negative, continuous, admissible functions $\lambda$ on $\r^+$ with $\lambda(0)>0$ when $\alpha>0$ and further with finite and connected support when $\alpha=0$. The $\text{BESQ}^{2\alpha}$ process is in $\widetilde{\mathfrak{R}}$ a.s..

For $f\in \widetilde{\mathfrak{R}}$, define the following change of scale $\eta_f$ and change of time $K_f$ associated to a deterministic continuous process $R$ on $\r^+$ starting from $0$:
\begin{align*}
&\eta_f(y)=\int_0^yf(x)^{-1}\d x, y\in [0, \mathfrak{d}_f),\\
& K_f(t)=K^R_f(t)=\int_0^t \big(f\circ\eta_{f}^{-1}(R_s)\big)^2\d s,\ t\in[0,H^{R}_{\mathfrak{d}_f}).
\end{align*}

Set $\Phi(R_{[0,a)},f):=\left(\eta^{-1}_f\big(R_{K_f^{-1}(t)}\big):t\in [0,K_f(a))\right)$ for $0<a\le H^{R}_{\mathfrak{d}_f}$. 

For any $\lambda\in \widetilde{\mathfrak{R}}$, define the self-repelling diffusion $W^\lambda$ as follows:
\begin{itemize}
    \item when $\alpha=0$, let $\cZ^{(1)}$,  $\cZ^{(2)}$ and $\mathcal{J}^{2,2}$ be three independent processes such that $\cZ^{(1)}$ and  $\cZ^{(2)}$ are identically distributed as $\cZ$ and $\mathcal{J}^{2,2}$ is a diffusion on $[0,1]$ with infinitesimal generator $2x(1-x)\d^2x+2(1-2x)\d x$. Set $\lambda^{(1)}=\lambda(x)\mathcal{J}^{2,2}_{\eta_\lambda(x)}$, $\lambda^{(2)}(x)=\lambda(x)-\lambda^{(1)}(x)$ and $t^{(i)}=K^{Z^{(i)}}_{\lambda^{(i)}}(\infty)$\footnote{It holds that $T^{Z^{(i)}}=\infty$ (Cf. \cite{warren98}).} $(i=1,2)$. Define
\begin{align}\label{zalphaequal}
W^\lambda_t:=
\begin{cases}
\Phi(\cZ^{(1)},\lambda^{(1)},\infty)(t), &\text{$0\le t\le t^{(1)}$},\\
    \Phi(\cZ^{(2)},\lambda^{(2)},\infty)(t^{(1)}+t^{(2)}-t), &t^{(1)}< t\le t^{(1)} +t^{(2)};
\end{cases}
\end{align}
\item when $\alpha>0$, set
\begin{equation}\label{zalphalarge}
    W^\lambda:=
\Phi(\cZ,\lambda).
\end{equation}
\end{itemize}
\begin{proposition}[\cite{Aidekon,warren98}]\label{rayline}
Let $(\widetilde{\phi}^{(0)},B_{[0,\tau^B_u]},\widetilde{\phi}^{(u)})$ be a Ray-Knight triple associated to $B$. For any $\lambda\in\widetilde{\mathfrak{R}}$,
$W^\lambda$ has the conditional law of $B_{[0,\tau^B_u]}$ given $\widetilde{\phi}^{(u)}=\lambda$.
\end{proposition}

We call $W^\lambda$ the self-repelling diffusion on $\r^+$.  
Recall the setting in \S\ref{scalelimit}. Let $\Omega_k$ be the collection of right-continuous, minimal, nearest-neighbor path on $\N_k$. Fix  $\lambda\in\widetilde{\mathfrak{R}}$.
For $k\ge 1$, $x,y\in \N_k$ with $x\sim y$ and $\omega\in \Omega_k$ with $T^\omega>t$, set
\begin{align}\label{lambdak}
\begin{split}
  &\Lambda^{k}_t(x,\omega)=\lambda(x)-2^k\int_0^t 1_{\{\omega_s=x\}}\d s,\\
   &\varphi_t^{k}(xy,\omega)=2C^k_{xy}\sqrt{\Lambda_t^{k}(x,\omega) \Lambda_t^{k}(y,\omega)}.
    \end{split}
\end{align}
It is easy to check that $X^{\lambda,(k)}$ is a jump process on $\N_k$ with jump rates $r^{\lambda,(k)}(x,y)$ and the similar resurrect mechanism and lifetime as before, where
\begin{small}
\begin{align}\label{xrate}
\begin{split}
   r^{\lambda,(k)}(x,y)&:=\left\{\begin{aligned}
    &2^{2k-1}\sqrt{\dfrac{\Lambda^{k}_t(y)}{\Lambda^{k}_t(x)}}\cdot\dfrac{I_{\alpha-1}(\varphi^{k}_t(xy))}{I_{\alpha}(\varphi^{k}_t(xy))},&\text{ if $x=\mathfrak{p}(y)$};\\
  &2^{2k-1}\sqrt{\dfrac{\Lambda^{k}_t(y)}{\Lambda^{k}_t(x)}}\cdot\dfrac{I_{\alpha}(\varphi^{k}_t(xy))}{I_{\alpha-1}(\varphi^{k}_t(xy))},&\text{ if $y=\mathfrak{p}(x)$};\\
   \end{aligned}\right.
  \end{split}
\end{align}
\end{small}

The proof of Theorem \ref{convergence}  follows the same routine as  \cite{LupuEJP657}. So we omit the details here.

\begin{appendix}

\section{Processes with terminated jump rates}\label{selfinteracting}
Let $G=(V,E)$ be a finite or countable graph with finite degree. All the processes considered in this part are assumed to be right-continuous, minimal, nearest-neighbour jump processes on $G$ with a finite or infinite lifetime. We consider the process to stay at a cemetery point $\Delta$ after the lifetime. The collection of all such sample paths  is denoted by $\Omega$, and $\Omega_\infty:=\{\omega\in \Omega:\omega \text{ does not explode}\}$. Let $\left\{\mathscr{F}_t:t\ge 0\right\}$ be the natural filtration of the coordinate process on $\Omega_\infty$. 

Let us introduce some notations that will be used throughout this part. For $0\le a<b\le \infty$, $[a,b\rangle$ represents the interval $[a,b]$ or $[a,b)$ according as $b<\infty$ or $b=\infty$. Let $l\ge 0$, $t\in(0,\infty]$, $\{t_i\}_{1\le i\le l}$ be positive real numbers and $\{x_i\}_{0\le i\le l}$ be vertices in $V$, such that $0=:t_0<t_1<\cdots<t_l\le t_{l+1}:=t$ and $x_0\sim x_1\sim \cdots \sim x_l$. Denote by $\sigma(x_0\overset{t_1}{\rightarrow}x_1\overset{t_2}{\rightarrow}\cdots\overset{t_l}{\rightarrow}x_l\overset{t}{\rightarrow})$ the function: $[0,t\rangle\rightarrow V$, that equals $x_i$ on $[t_i,t_{i+1})$ (resp. $[t_i,t_{i+1}\rangle$) for $i=0,\cdots,l-1$ (resp. $i=l$).  

\begin{definition}[LSC stopping time]\label{defdefdef}
Let $T$ be a stopping time with respect to $\left(\mathscr{F}_t:t\ge 0\right)$, such that it is lower semi-continuous under the Skorokhod topology. We simply say $T$ is a \textit{LSC stopping time}. The notation $1_{\{t<T(\omega_{[0,t]})\}}$ has the natural meaning when $T$ is a stopping time. We further say $T$ is \textit{regular} if for any $\omega$ with $t<T(\omega)$ and $\omega_t=x$, there exists $l=l(\omega_{[0,t]})>0$, such that for any $0<s\le l$ and $y\sim x$, it holds that $T\big(\omega_{[0,t]}\circ \sigma(x\overset{s}{\rightarrow}y\overset{\infty}{\rightarrow})\big)\ge t+l$. Intuitively, if we regard $T(\omega)$ as the lifetime of $\omega$, then the above statement means that if $\omega$ is alive at time $t$, then as long as it has exactly one jump during $[t,t+l]$, it is bound to stay alive until time $t+l$. 
\end{definition}

\begin{definition}[Terminated jump rates]
Given a LSC stopping time $T$. Denote
\[\mathscr{D}_T:=\big\{(t,x,y,\omega)\in \r^+\times V\times V\times \Omega_\infty:0\le t< T(\omega),\,x\sim y,\,\omega_t=x\big\}.\]
A function $r=r_t(x,y)=r(t,x,y,\omega)$: $\mathscr{D}_T\rightarrow \r^+$ is called (a family of) \textit{$T$-terminated jump rates} if (1) it is continuous with respect to the product topology, where $\Omega_\infty$ is equipped with the Skorokhod topology; (2) for any $x\sim y$, the process $\big(r_t(x,y)1_{\{t<T\}}:t\ge 0\big)$ is adapted to $(\mathscr{F}_t)$. For such $r$, we also write $r_t(x,y,\omega)=r_t(x,y,\omega_{[0,t]})$ on $\{t<T(\omega)\}$. 
\end{definition}

From now on, we always assume that $T$ is a LSC stopping time and $r$ is $T$-terminated jump rates.
 
\begin{definition}[Processes with terminated jump rates]\label{processtermjump}
A process $Z=(Z_s:0\le s<T^Z)$ is said to have jump rates $r$ if for any $t\ge 0$, conditionally on $t<T^Z$ and $\left(Z_s:0\le s \le t\right)$ with $Z_t=x$, 
\begin{itemize}
    \item[(R1)] the probability that the first jump of $Z$ after time $t$ occurs in $[t,t+\Delta t]$ and the jump is to a neighbour $y$ is $r_t(x,y,Z_{[0,t]})\Delta t+o(\Delta t)$, where $o(\Delta t)$ depends on the conditioned path $Z_{[0,t]}$,
\end{itemize}
and the process finally stops at time $T^Z$. Here $T^Z=T^Z_0\wedge T^Z_\infty$ with
\begin{align*}
    T^Z_0&:=\sup\{t\ge 0: t<T(Z_{[0,t]})\};\\
    T^Z_\infty&:=\sup\{t\ge 0:Z_{[0,t]}\text{ has finitely many jumps}\}.
\end{align*}
\end{definition}
\begin{remark}\label{R1R2equal}
In the above definition, if it holds that $T$ is regular, then the condition (R1) can be replaced by (R2) below without affecting the law of the defined process (see Remark \ref{R123equal} for details).
\begin{itemize}
    \item[(R2)]  the probability of a jump of $Z$ from $x$ to a neighbour $y$ in $[t,t+\Delta t]$ is $\allowbreak r_t(x,y,Z_{[0,t]})\Delta t+o(\Delta t)$, where $o(\Delta t)$ depends on the conditioned path $Z_{[0,t]}$.
\end{itemize}
\end{remark}

The following renewal property is direct from the definition.
\begin{proposition}[Renewal property]\label{markovproperty}
Suppose $Z$ is a process with jump rates $r$. Then for any $t\ge 0$, conditionally on $t<T^Z$ and $(Z_s:0\le s\le t)$, the process $(Z_{t+s}:0\le s<T^Z-t)$ is a process with jump rates $r^\prime_u(x,y,\omega)=r_{t+u}(x,y,Z_{[0,t)}\circ \omega)$ starting from $Z_t$, where $Z_{[0,t)}\circ \omega$ is a function that equals $Z_{[0,t)}$ on $[0,t)$ and equals $\omega(\cdot-t)$ on $[t,\infty)$. 
\end{proposition}

 
Now we tackle the basic problem: the existence and uniqueness of the processes with jump rates $r$. Let $\mathscr{B}$ be the $\sigma$-field on $\Omega$ generated by the coordinate maps.
\begin{theorem}\label{existunique}
For any $x_0\in V$, there exists a process with jump rates $r$ starting from $x_0$. Moreover, if $T$ is regular, then all such processes have the same distribution on $\left(\Omega,\mathscr{B}\right)$.
\end{theorem}

\begin{proof}[Proof of the existence part of Theorem \ref{existunique}] 
Similar to the case of CTMC, we can use a sequence of i.i.d. exponential random variables to construct the process. 
Precisely, let $\left(\gamma_x:x\sim x_0\right)$ be a family of i.i.d. exponential random variables with parameter $1$. For $x\sim x_0$, we set 
\[u_{x_0x}(t)=\int_0^t r_s(x_0,x,\sigma(x_0\overset{\infty}{\rightarrow}))\d s \quad\left(0\le t<T(\sigma(x_0\overset{\infty}{\rightarrow})\right)\]
and $\Gamma_{x_0x}=(u_{x_0x})^{-1}(\gamma_x)$, where $(u_{x_0x})^{-1}$ is the right-continuous inverse and if $\gamma_x\ge u_{x_0x}(T(\sigma(x_0\overset{\infty}{\rightarrow})-)$, then $\Gamma_{x_0x}:=\infty$. The process starts at $x_0$. If $\Gamma_{x_0}:=\min_{x\sim x_0}\Gamma_{x_0x}=\infty$, the process stays at $x_0$ until the lifetime $T(\sigma(x_0\overset{\infty}{\rightarrow}))$. Otherwise, i.e., $\Gamma_{x_0}<\infty$, it jumps at time $J_1:=\Gamma_{x_0}$ to $x_1$, which is the unique $x\sim x_0$ such that $\Gamma_{x_0x}=\Gamma_{x_0}$. For the second jump, the protocol is the same except that $x_0$ and $\sigma(x_0\overset{\infty}{\rightarrow})$ are replaced by $x_1$ and $\sigma(x_0\overset{J_1}{\rightarrow}x_1\overset{\infty}{\rightarrow})$ respectively.

In the same way that one verifies the jump rates of a CTMC constructed from a sequence of i.i.d. exponential random variables, it is simple to check that the process constructed above has jump rates $r$.
\end{proof}
The proof of the uniqueness part is delayed to \S\ref{pathprobsection} after the introduction of path probability.

\begin{remark}\label{rigorousdef}
In this remark, we give the rigorous definitions of repelling processes presented in \S2$\sim$4. We only construct the process with law $\p^\lambda_{x_0}$ (defined in \S\ref{s:inversionrk}) and it is similar for the others. For $\lambda\in \mathfrak{R}$, let $T^\lambda(\omega):=\sup\{t\ge 0:\Lambda_t(\omega_{t-},\omega)>0\}$ and consider
$r^\lambda$ (defined as \eqref{xxrate}) restricted to $\mathscr{D}_{T^\lambda}$. Intuitively, $T^\lambda(\omega)$ represents the time when the process $\omega$ uses up the local time at some vertex.
We can readily check that $T^\lambda$ is a regular LSC stopping time and $r^\lambda$ is $T^\lambda$-terminated jump rates. We first run a process $Z^1$ with jump rates $r^\lambda$ starting from $x_0$ up to $T^{Z^1}$. If the death is due to the exhaustion of local time at some vertex $x\neq x_0$, i.e. $T^{Z^1}=T^{Z^1}_0$ and $Z^1_{T^{Z^1}-}=x$, we record the remaining local time at each vertex, say $\varrho$. Then we resurrect the process by letting it jump to $\mathfrak{p}(x)$ and running an independent process $Z^2$ with jump rates $r^\varrho$ starting from $\mathfrak{p}(x)$ up to $T^{Z^2}$. After the death, we resurrect the process again. This continues until it explodes or exhausts the local time at $x_0$. This procedure defines a process, the law of which is defined to be $\p^\lambda_{x_0}$.
By Theorem \ref{existunique}, such a process always exists and the law is already determined by the definition. 

As the analysis in \S\ref{main666}, when $\alpha>0$, $Z^1$ cannot die at $x\neq x_0$. So in this case, we actually do not need the resurrect procedure, and $Z^1$ up to $T^{Z^1}$ has the law $\p^\lambda_{x_0}$.

\end{remark}

Based on Theorem \ref{existunique}, if $T$ is regular, when considering a process with terminated jump rates, we can always focus on its special realization presented in the previous proof of the existence part. The following three corollaries are immediately derived with this perspective. 
\begin{corollary}
Suppose $Z$ is a process with jump rates $r$. Let $J_1$ be the first jump time of $Z$. Then
\begin{align}\label{J1pathprob}
\begin{aligned}
    &\p\left(J_1\in\d t,\,Z_{J_1}=x\right)\\
    &=1_{\{t<T(\sigma(x_0\overset{\infty}{\rightarrow}))\}}\exp\Big(-\int_0^t \sum_{y:y\sim x_0}r_s(x_0,y,\sigma(x_0\overset{\infty}{\rightarrow}))\d s\Big)\cdot r_t(x_0,x,\sigma(x_0\overset{\infty}{\rightarrow}))\d t,    
\end{aligned}
\end{align}
\end{corollary}

\begin{corollary}[Strong renewal property]\label{selfstrong}
Suppose $Z$ is a process on $G$ starting from $x_0$ with jump rates $r$. Then for any stopping time $S$ with respect to the natural filtration of $Z$, conditionally on $S<T^Z$ and $(Z_s:0\le s\le S)$, the process $(Z_{S+s}:0\le s\le T^Z)$ is a process with jump rates $r^\prime_t(x,y,\omega)=r_{t+S}(x,y,Z_{[0,S)}\circ \omega)$ starting from $Z_S$.
\end{corollary}
The proof of the above strong renewal property is almost a word-by-word copy of the proof of the strong Markov property of CTMC presented in \cite[Theorem 6.5.4]{norris1998markov}.

The next corollary gives a more accurate bound of the probability in (R1) and (R2). We only state in terms of the event in (R2). For $t\ge 0$ and $\sigma\in \Omega_t:=\{\omega_{[0,t]}:\omega\in \Omega_\infty\}$, we use $\sigma^\rightarrow$ to represent the function in $\Omega_\infty$ that equals $\sigma$ on $[0,t]$ and stays at $\sigma_t$ after time $t$.
\begin{corollary}\label{accuratebound}
Suppose $Z$ is a process with jump rates $r$ starting from $x_0$. Then any for $\sigma\in \Omega_t$ with $\sigma_t=x$ and $t<T(\sigma)$, if we denote $R(z)=\exp\Big(-\int_t^{t+\Delta t}r_s(x,y,\sigma^\rightarrow)\d s\Big)$ for $z\sim x$, it holds that
\begin{align*}
    &\p\left(\text{there is a jump of $Z$ from $x$ to $y$ in }[t,t+\Delta t]\,\Big\vert \,Z_{[0,t]}=\sigma\right)\\
    &\in\bigg[\Big(\prod_{z:z\sim x,\, z\neq y}R(z)\Big)\cdot\big(1-R(y)\big),\ 1-R(y)\bigg].
\end{align*}
\end{corollary}

\subsection{Path probability}\label{pathprobsection}
\ 

Assume $T$ is regular. Let $Z$ be a process with jump rates $r$ starting from $x_0$. For any Borel subsets $\{D_i\}_{1\le i\le l}$ of $\r^+$, $t\in (0,\infty]$ and $\{x_i\}_{1\le i\le l}$ with $x_0\sim x_1\sim \cdots\sim x_l$, denote
\begin{align}\label{Adef666}
\begin{aligned}
   A\big(x_0\overset{D_1}{\rightarrow}x_1\overset{D_2}{\rightarrow}\cdots\overset{D_l}{\rightarrow}x_l\overset{t}{\rightarrow}\big):=\Big\{\sigma(x_0\overset{s_1}{\rightarrow}x_1\overset{s_2}{\rightarrow}\cdots\overset{s_l}{\rightarrow}x_l\overset{t}{\rightarrow}):0<s_1<&\\
    \cdots<s_l\le t \text{ and } s_i\in D_i\ \forall\,1\le i\le l\Big\}&. 
\end{aligned}
\end{align}
The goal of this part is to calculate the path probability
\[\p\Big(Z_{[0,t]}\in A\big(x_0\overset{D_1}{\rightarrow}x_1\overset{D_2}{\rightarrow}\cdots\overset{D_l}{\rightarrow}x_l\overset{t}{\rightarrow}\big)\Big).\]
The key point of the calculation is the following lemma. At first glance, the lemma seems obvious. We mention that the main difficulty is due to the dependence of $o(\Delta t)$ in (R1) on the conditioned path $Z_{[0,t]}$.
\begin{lemma}\label{twojumporder}
Conditionally on $t<T^Z$ and $(Z_s:0\le s\le t)$ with $Z_t=x$, the probability that $Z$ has at least $2$ jumps during $[t,t+\Delta t]$ and the first jump is to $y$ is $O((\Delta t)^2)$, where $O((\Delta t)^2)$ depends on the conditioned path.
\end{lemma}
\noindent {\it Remark}.
Without the regularity of $T$, it may happen that the conditional probability of two jumps in $[t,t+\Delta t]$ is comparable to $\Delta t$. For example, we consider a path-dependent Poisson process as follows. Starting from $0$, it jumps to $1$ with rate $1$. Its jump rate at $1$ is given by: 
\[r_s\big(1,2,\sigma(0\overset{t}{\rightarrow}1\overset{\infty}{\rightarrow})\big):=\frac{1}{2t-s}~\text{ for $t>0$ and $t\le s\le 2t$}.\]
Let us consider the special realization presented in Theorem \ref{existunique} with the above rates. Then it holds that conditionally on the process jumping from $0$ to $1$ at time $t$, it is bound to jump from $1$ to $2$ before time $2t$. So the probability of two jumps in $[0,\Delta t]$ is no less than $1-e^{-\Delta t/2}$.

\

Before the proof, we mention a simple result. Recall the notation $\sigma^\rightarrow$. It is easy to see that
given $s> t\ge 0$ and $\sigma\in \Omega_t$ with $\sigma_t=x$ and $s< T(\sigma^\rightarrow)$,
\begin{align}\label{holdingprob}
\begin{aligned}
\p\left(Z_v=x\text{ on }[t,s]\,\big\vert \,Z_{[0,t]}=\sigma\right)=\exp\Big(-\int_t^s\sum_{y:y\sim x}r_u(x,y,\sigma^{\rightarrow})\d u\Big).
\end{aligned}
\end{align}
To this end, one can consider the left-hand side of \eqref{holdingprob} as a function of $s$ and formulate a differential equation.

\begin{proof}[Proof of Lemma \ref{twojumporder}]
The renewal property (in Proposition \ref{markovproperty}) enables us to consider only the case of $t=0$. Let $J_1$ and $J_2$ be the first and second jump time of $Z$ respectively, and $B=B(y):=\big\{J_1,J_2\in[0,\Delta t],\, Z_{J_1}=y\big\}$ for $y\sim x_0$. The event $B$ can be divided according to the jump times as follows. For $n\ge 1$, define
\[B_{j,n}:=\Big\{J_1\in \Big(\frac{j-1}{2^n}\Delta t,\frac{j}{2^n}\Delta t\Big],\,J_2\in\Big(\frac{j}{2^n}\Delta t,\Delta t\Big],\,Z_{J_1}=y\Big\}\ (1\le j\le 2^n-1).\]
Set $B_n:=\bigcup_{j=1}^{2^n-1} B_{j,n}$. Then it holds that $B_n\uparrow B$. So it suffices to show that there exists a constant $C$ independent of $n$, such that
\[\p(B_n)\le C(\Delta t)^2,\text{ for all }n\ge 1.\]
For any $j$ and $n$ with $1\le j\le 2^n-1$, let $B_{j,n}^{(1)}$ be the event that $J_1\in \Big(\frac{j-1}{2^n}\Delta t,\frac{j}{2^n}\Delta t\Big]$ and $Z_u=y\text{ for }u\in \Big[J_1,\frac{j}{2^n}\Delta t\Big]$. Then
\begin{align*}
    \p\left(B_{j,n}\right)= \p\Big(J_2\in \Big(\frac{j}{2^n}\Delta t,\Delta t\Big]\,\Big\vert \,B_{j,n}^{(1)}\Big)\cdot \p\Big(B_{j,n}^{(1)}\Big).
\end{align*}
For the first probability on the right-hand side, observe that if we  further condition on $J_1$, then we have conditioned on the whole path $(Z_s:0\le s\le \frac{j}{2^n}\Delta t)$. By the regularity of $T$, for $\Delta t$ sufficiently small, $T(\sigma(x_0\overset{u}{\rightarrow}y\overset{\infty}{\rightarrow}))>\Delta t$ for any $0<u\le \Delta t$. Then it follows from \eqref{holdingprob} that the conditional probability of $J_2\in \Big(\frac{j}{2^n}\Delta t,\Delta t\Big]$ is
\begin{align}\label{doublejump}
    1-\exp\bigg(-\int_{(\frac{j}{2^n}\Delta t,\Delta t]}\sum_{z:z\sim y}r_v\big(y,z,(Z_{[0,\frac{j}{2^n}\Delta t]})^\rightarrow\big)\d v\bigg)\le C\Delta t,
\end{align}
where
    $C=\sup\Big\{ \sum_{z:z\sim y}r_v(y,z,\sigma(x_0\overset{u}{\rightarrow}x_1\overset{\infty}{\rightarrow})):0\le u\le v\le \Delta t\Big\}$,
which is finite by the continuity of $r$. The same bound also holds for $\p\Big(J_2\in \Big(\frac{j}{2^n}\Delta t,\Delta t\Big]\,\Big\vert \,B_{j,n}^{(1)}\Big)$. So
\begin{align*}
    \p(B_n)&=\sum_{j=1}^{n-1}\p(B_{j,n})\le C\Delta t\sum_{j=1}^{n-1}\p\Big(B_{j,n}^{(1)}\Big)\le C\Delta t\cdot\p(J_1\in[0,\Delta t],\,Z_{J_1}=y)\\
    &=C\Delta t\cdot \left( r_0(x_0,y,\sigma(x_0\overset{\infty}{\rightarrow}))\Delta t+o(\Delta t)\right)\le C^\prime (\Delta t)^2.
\end{align*}
That completes the proof.
\end{proof}


\begin{corollary}\label{exactlyonejump}
Conditionally on $t<T^Z$ and $(Z_s:0\le s\le t)$ with $Z_t=x$, the probability that $Z$ has exactly one jump during $[t,t+\Delta t]$ and the jump is to $y$ is $r_t(x,y,Z_{[0,t]})\Delta t+o(\Delta t)$, where $o(\Delta t)$ depends on the conditioned path.
\end{corollary}

Let us start the calculation of path probability. The case of $l=0$ has been covered by \eqref{holdingprob}. For $l\ge 1$, Corollary \ref{exactlyonejump} enables us to use the methods of formulating differential equations to calculate path probabilities. Fix $x_1\sim x_0$ and $0<t_1<t_2\le \infty$
with $t_2<T(\sigma(x_0\overset{t_1}{\rightarrow}x_1\overset{\infty}{\rightarrow}))$. By the lower semi-continuity of $T$, there exists
$0<u<t_2-t_1$, such that for any $0\le s<u$, $t_2<T\left(\sigma(x_0\overset{t_1+s}{\rightarrow}x_1\overset{\infty}{\rightarrow})\right)$. For such $s$, let us calculate
\[q(s):=\p\Big(Z_{[0,t_2]}\in A\big(x_0\overset{[t_1,t_1+s]}{\rightarrow}x_1\overset{t_2}{\rightarrow}\big)\Big).\]
For $0<\Delta s<u-s$, 
\begin{align*}
    q(s+\Delta s)-q(s)=\p\Big(Z_{[0,t_2]}\in A\big(x_0\overset{[t_1+s,t_1+s+\Delta s]}{\rightarrow}x_1\overset{t_2}{\rightarrow}\big)\Big)=:q_1q_2q_3,
\end{align*}
where $q_1=\p(E_1)$, $q_2=\p(E_2\,\vert \,E_1)$, $q_3=\p(E_3\,\vert \,E_1\cap E_2)$ with
\begin{align*}
    E_1&=\{Z_u=x_0\text{ for } u\in[0,t_1+s]\},\\
    E_2&=
        \{Z \text{ has exactly one jump during }[t_1+s,t_1+s+\Delta s]\text{ and the jump is to }x_1\},\\
    E_3&=\{Z_u=x_1\text{ for }u\in [t_1+s+\Delta s,t_2]\}.
\end{align*}
It holds that
\begin{align}\label{q1q2q3}
    \left\{\begin{aligned}
    &q_1=\exp\Big(-\int_0^{t_1+s} \sum_{y:y\sim x_0}r_v\big(x_0,y,\sigma(x_0\overset{\infty}{\rightarrow})\big)\d v\Big);\\
    &q_2=r_{t_1+s}\big(x_0,x_1,\sigma(x_0\overset{\infty}{\rightarrow})\big)\cdot\Delta s+o(\Delta s);\\
    &q_3=\exp\Big(-\int_{t_1+s}^{t_2}\sum_{y:y\sim x_1}r_v\big(x_1,y,\sigma(x_0\overset{t_1+s}{\rightarrow}x_1\overset{\infty}{\rightarrow})\big)\d v\Big)+o(1), 
    \end{aligned}\right.
\end{align}
where $q_1$ and $q_2$ follows from \eqref{holdingprob} and Corollary \ref{exactlyonejump} respectively. 
For $q_3$, it suffices to note that for any $0\le h\le \Delta s$, conditionally on $Z_{[0,t_1+s+\Delta s]}=\sigma(x_0\overset{t_1+s+h}{\rightarrow}x_1\overset{t_1+s+\Delta s}{\rightarrow})$, the probability of $E_3$ is
\[\exp\Big(-\int_{t_1+s+\Delta s}^{t_2}\sum_{y:y\sim x_1}r_v\big(x_1,y,\sigma(x_0\overset{t_1+s+h}{\rightarrow}x_1\overset{\infty}{\rightarrow})\big)\Big)\d v,\]
which together with the continuity of $r$ easily leads to the probability $q_3$.

By further considering the case $\Delta s<0$, we can formulate a differential equation, the solution of which gives: for $0\le s<u$,
\begin{align*}
    q(s)&=\int_{t_1}^{t_1+s}\exp\Big(-\int_0^{t_2}\sum_{y:y\sim \sigma_v}r_v(\sigma_v,y,\sigma)\d v\Big)\cdot r_{s^\prime}(x_0,x_1,\sigma)\d s^\prime,
\end{align*}
where $\sigma=\sigma(x_0\overset{s^\prime}{\rightarrow}x_1\overset{\infty}{\rightarrow})$ which varies with $s^\prime$. Observe that the lower semi-continuity of $T$ implies that $\{s^\prime\in [0,t_2]:t_2<T(\sigma(x_0\overset{s^\prime}{\rightarrow}x_1\overset{t_2}{\rightarrow}))\}$ is a relatively open subset of $[0,t_2]$. So we can readily deduce that for any Borel subset $D$ in $\r^+$ and $x_1\in V$,
\begin{align*}
\begin{aligned}
    &\p\left(Z_{[0,t]}\in A\left(x_0\overset{D}{\rightarrow}x_1\overset{t}{\rightarrow}\right)\right)\\
    &=\int_D 1_{\{s^\prime<t<T(\sigma)\}} \exp\Big(-\int_0^{t}\sum_{y:y\sim \sigma_v}r_v(\sigma_v,y,\sigma)\d v\Big)\cdot r_{s^\prime}(x_0,x_1,\sigma)\d s^\prime.
\end{aligned}
\end{align*}


Similarly, we can inductively check that for Borel subsets $D_i$ in $\r^+$ and $x_i\in V$ ($1\le i\le l$), 
\begin{align}\label{Omegapathprob2}
\begin{aligned}
    \p\left(Z_{[0,t]}\in A_t\right)&=\idotsint_{D_1\times\cdots\times D_l}1_{\{s_1<\cdots<s_l<t<T(\sigma)\}}\cdot\prod_{j=1}^l r_{t_j}(x_{j-1},x_j,\sigma)\\
    &\quad\cdot\exp\Big(-\int_0^t\sum_{y:y\sim \sigma_v}r_v(\sigma_v,y,\sigma)\d v\Big)\cdot\prod_{j=1}^l\d s_j,
\end{aligned}
\end{align}
where $\sigma=\sigma(x_0\overset{s_1}{\rightarrow}x_1\overset{s_2}{\rightarrow}\cdots\overset{s_l}{\rightarrow}x_l\overset{t}{\rightarrow})$ and $A_t=A\big(x_0\overset{D_1}{\rightarrow}x_1\overset{D_2}{\rightarrow}\cdots\overset{D_l}{\rightarrow}x_l\overset{t}{\rightarrow}\big)$.

When $(x_i)$, $(D_i)$ and $t$ vary, all the sets $A_t=A\big(x_0\overset{D_1}{\rightarrow}x_1\overset{D_2}{\rightarrow}\cdots\overset{D_l}{\rightarrow}x_l\overset{t}{\rightarrow}\big)$ constitute a $\pi$-system. The $\sigma$-field generated by these sets on $(\Omega,\mathscr{B})$ contains sets in the form $\{\omega:\omega_s=x\}$ ($x\in V$, $s\ge 0$) and hence equals $\mathscr{B}$. So \eqref{Omegapathprob2} determines the law of $Z$. That completes the proof of the uniqueness part of Theorem \ref{existunique}.

\begin{remark}\label{R123equal}
Observe that if we replace (R1) by (R2) in the Definition \ref{processtermjump}, then following the same routine, we obtain the same path probability \eqref{Omegapathprob2}. In fact, the only difference in the proof is that in \eqref{holdingprob} `$=$' should be replaced by `$\ge$', and the later proofs still work with minor modification. This easily leads to the statement in Remark \ref{R1R2equal}.
\end{remark}

\end{appendix}

\begin{acks}[Acknowledgments]
The authors are grateful to Prof. A\"{i}d\'{e}kon Elie for introducing this project and inspiring discussions. The authors
would also like to thank Prof. Jiangang Ying and Shuo Qin  for helpful discussions and valuable suggestions. 
\end{acks}

\begin{funding}
The first  author is partially supported by the Fundamental Research Funds for the Central Universities. The second  author is partially supported by NSFC, China (No. 11871162).
\end{funding}

\bibliographystyle{imsart-number} 
\bibliography{rayknight}       





\end{document}